\documentclass[12pt]{amsart}

\usepackage{amsmath,amssymb,enumerate,mathtools}
\newtheorem{theorem}{Theorem}[section]
\newtheorem{claim}{}[theorem]
\newtheorem{lemma}[theorem]{Lemma}

\newtheorem{conjecture}[theorem]{Conjecture}
\theoremstyle{definition}

\newcommand{\bF}{\mathbb F}

\newcommand{\bZ}{\mathbb Z}

\newcommand{\bQ}{\mathbb Q}

\newcommand{\cF}{\mathcal{F}}

\newcommand{\cH}{\mathcal{H}}
\newcommand{\cL}{\mathcal{L}}

\newcommand{\fE}{\mathbf{E}}
\newcommand{\fM}{\fE}
\newcommand{\cM}{\mathcal{M}}
\newcommand{\cO}{\mathcal{O}}
\newcommand{\cP}{\mathcal{P}}

\newcommand{\cX}{\mathcal{X}}

\newcommand{\bE}{\mathbf{E}}
\newcommand{\bI}{\mathbf{I}}

\newcommand{\lex}{\mathrm{lex}}
\DeclareMathOperator{\si}{si}

\DeclareMathOperator{\cl}{cl}

\DeclareMathOperator{\PG}{PG}
\DeclareMathOperator{\Ex}{Ex}
\DeclareMathOperator{\GF}{GF}
\DeclareMathOperator{\AG}{AG}
\newcommand{\wh}{\widehat}
\newcommand{\elem}{\varepsilon}
\newcommand{\del}{\!\setminus\!}
\newcommand{\con}{/}

\newcommand{\whp}{\widehat{\varphi}}

\allowdisplaybreaks[1]

\begin{document}
\sloppy

\title[Exponential growth rates]{The Densest Matroids in Minor-Closed Classes with Exponential Growth Rate}

\author[Geelen]{Jim Geelen}
\address{Department of Combinatorics and Optimization,
University of Waterloo, Waterloo, Canada}
\thanks{ This research was partially supported by a grant from the
Office of Naval Research [N00014-10-1-0851].}

\author[Nelson]{Peter Nelson}

\begin{abstract}
 The \emph{growth rate function} for a nonempty minor-closed class of matroids $\cM$ is the function $h_{\cM}(n)$ whose value at an integer $n \ge 0$ is defined to be the maximum number of elements in a simple matroid in $\cM$ of rank at most $n$. Geelen, Kabell, Kung and Whittle showed that, whenever $h_{\cM}(2)$ is finite, the function $h_{\cM}$ grows linearly, quadratically or exponentially in $n$ (with base equal to a prime power $q$), up to a constant factor. 
 
 We prove that in the exponential case, there are nonnegative integers $k$ and $d \le \tfrac{q^{2k}-1}{q-1}$ such that $h_{\cM}(n) = \frac{q^{n+k}-1}{q-1} - qd$ for all sufficiently large $n$, and we characterise which matroids attain the growth rate function for large $n$. We also show that if $\cM$ is specified in a certain `natural' way (by intersections of classes of matroids representable over different finite fields and/or by excluding a finite set of minors), then the constants $k$ and $d$, as well as the point that `sufficiently large' begins to apply to $n$, can be determined by a finite computation. 
\end{abstract}

\date{\today}

\maketitle

\section{Introduction}

We write $\elem(M)$ for the number of points in a matroid $M$, so $\elem(M) = |M|$ for every simple matroid $M$. For a nonempty minor-closed class of matroids $\cM$, the \emph{growth rate function $h_{\cM}(n) \colon \bZ_{\ge 0} \to \bZ \cup \{\infty\}$} for $\cM$ is defined for all $n \ge 0$ by \[h_{\cM}(n) = \max\{\elem(M)\colon M \in \cM, r(M) \le n\}.\]  If $\cM$ contains all rank-$2$ uniform matroids, then clearly $h_{\cM}(n) = \infty$ for all $n \ge 2$. Otherwise, $h_{\cM}(n)$ is controlled up to a constant factor by the following theorem of Geelen, Kabell, Kung and Whittle [\ref{gkw}]:

\begin{theorem}[Growth Rate Theorem]\label{grt}
	Let $\cM$ be a minor-closed class of matroids not containing all rank-$2$ uniform matroids. There exists an integer $\alpha$ such that either
	\begin{enumerate}
		\item\label{lin} $h_{\cM}(n) \le \alpha n$ for all $n \ge 0$, 
		\item\label{quad} $\binom{n+1}{2} \le h_{\cM}(n) \le \alpha n^2$ for all $n \ge 0$ and $\cM$ contains all graphic matroids, or
		\item\label{exp} there is a prime power $q$ such that $\frac{q^n-1}{q-1} \le h_{\cM}(n) \le \alpha q^n$ for all $n \ge 0$, and $\cM$ contains all $\GF(q)$-representable matroids. 
	\end{enumerate}
\end{theorem}

Classes of type (\ref{exp}) are \emph{(base-$q$) exponentially dense}. Our main result is an essentially best-possible refinement of condition (\ref{exp}), determining the precise value of each exponential growth rate function for all but finitely many $n$, and determining which large-rank matroids in $\cM$ attain this function. 

\begin{theorem}\label{main}
	Let $\cM$ be a base-$q$ exponentially dense minor-closed class of matroids. There are integers $k \ge 0$ and $d \in \{0,1,\dotsc,\tfrac{q^{2k}-1}{q^2-1}\}$ so that 
	\[h_{\cM}(n) = \tfrac{q^{n+k}-1}{q-1} - qd\]
	for all sufficiently large $n$. Moreover, if $M \in \cM$ has sufficiently large rank and satisfies $\elem(M) = h_{\cM}(r(M))$, then $M$ is, up to simplification, a $k$-element projection of a projective geometry over $\GF(q)$. 
\end{theorem}

By this we mean that $\si(M) \cong \si(M')$ for some matroid $M'$ obtained from a rank-$(r(M)+k)$ projective geometry over $\GF(q)$ by $k$ successive `extension-contraction' operations. The theorem was conjectured in [\ref{oldgrfpaper}].

	The `sufficiently large rank' condition in Theorem~\ref{main} is necessary in general, as the union of a base-$q$ exponentially dense minor-closed class $\cM$ with, say, the class of all $\GF(q')$-representable matroids of rank at most $t$ for some fixed $q' > q$, is base-$q$ exponentially dense, but  has a growth rate function that only adopts base-$q$ exponential behaviour for $n > t$. However, if $\cM$ is specified in some natural `finitary' way, then one might expect to compute $h_{\cM}(n)$ for every $n$. We prove a theorem showing that, in many cases, such a computation is possible in principle (as usual, $\Ex(\cO)$ denotes the class of matroids with no minor isomorphic to a matroid in $\cO$):
	
	\begin{theorem}\label{compute}
		Let $\cF$ be a finite set of finite fields and $\cO$ be a finite set of simple matroids. Let $\cM$ be the class of matroids in $\Ex(\cO)$ that are representable over all fields in $\cF$. If $\cM$ is base-$q$ exponentially dense and does not contain all truncations of $\GF(q)$-representable matroids, then there are computable nonnegative integers $k,d$ and $n_0$ such that $h_{\cM}(n) = \frac{q^{n+k}-1}{q-1} - qd$ for all $n \ge n_0$. 
	\end{theorem}
	
	Here, by \emph{computable} we mean that there is a Turing machine which, given some encoding of $\cF$ and $\cO$ as input, will output $k,d$ and $n_0$ in finite time. Since $h_{\cM}(n)$ can also be computed by exhaustion for all $n < n_0$, the above theorem shows that the entire growth rate function can be computed precisely for any such class.
	
	The insistence that the fields in $\cF$ are finite is artificial and just exists to avoid technicalities involving `specifying' an infinite field; if $\bF$-representability of a given matroid can be decided by a Turing machine for each $\bF \in \cF$, then the conclusion of Theorem~\ref{compute} still holds. The `truncation' condition, on the other hand, is necessary for our methods. Fortunately, this condition holds whenever $\cF \ne \varnothing$ or $\cO$ contains a co-line (since the truncation of a large circuit is a large co-line, which is not representable over a small finite field), so Theorem~\ref{compute} applies in both these natural cases.
	
	 If $\cM$ contains all  truncations of $\GF(q)$ representable matroids, then it appears that the densest matroids in $\cM$, though they are just small projections of projective geometries, can be `wild'. We conjecture that the problems this causes for our proof methods are fundamental, and that Theorem~\ref{computability} does not hold in full generality:
	
	\begin{conjecture}\label{undecidable}
		Let $\cO$ be a finite set of matroids such that $\Ex(\cO)$ is base-$q$ exponentially dense for some prime power $q$. It is undecidable to determine whether there exists an integer $k \ge 0$ such that $h_{\cM}(n) = \frac{q^{n+k}-1}{q-1}$ for all sufficiently large $n$.
	\end{conjecture}
	\section{Preliminaries}
	
	We mostly use the notation of Oxley [\ref{oxley}], but also write $|M|$ for $|E(M)|$ and $\elem(M)$ for $|\si(M)|$. Two matroids $M,N$ are equal \emph{up to simplification} if $\si(M) \cong \si(N)$. \emph{A simplification of $M$} is a $\si(M)$-restriction of $M$ (ie. any matroid obtained from $M$ by deleting all loops and all but one element from each parallel class.)
	
	An \emph{elementary projection} (also called a \emph{quotient}) of a matroid $M$ is a matroid $M'$ of the form $\wh M \con e$, where $e$ is a element of a matroid $\wh M$ that is not a loop or coloop, such that $M = \wh M \del e$. Thus, $r(M') = r(M) - 1$ and $E(M') = E(M)$. If $E(M)$ is the unique flat of $M$ that spans $e$ in $\wh M$, then $\wh M$ is the \emph{free extension} of $M$ by $e$, and $M'$ is the \emph{truncation} of $M$; we write $M' = T(M)$. 
	
	A \emph{$k$-element projection} of $M$ is a matroid obtained from $M$ by a sequence of $k$ elementary projections; it is easy to show that $M'$ is a $k$-element projection of $M$ if and only if there is a matroid $\wh M$ and a $k$-element independent independent set $K$ of $\wh M$ such that $M = \wh M \del K$ and $M' = \wh M \con K$.
	
	A collection $\cX$ of subsets in a matroid $M$ is \emph{skew} if $r_M(\cup_{X \in \cX} X) = \sum_{X \in \cX}r_M(X)$. A pair $(X,Y)$ of sets in $M$ is a \emph{modular pair} in $M$ if $r_M(X \cap Y) = r_M(X) + r_M(Y) - r_M(X \cup Y)$. A set $X$ is \emph{modular} in $M$ if it forms a modular pair with every flat of $M$. For example, every flat in a projective geometry is modular. Modularity gives a sufficient condition for a certain type of `sum' of two matroids to exist. If $M$ and $N$ are matroids, and the set $E(M) \cap E(N)$ is modular in $M$, then there is a unique matroid $M \oplus_m N$ with ground set $E(M) \cup E(N)$ that has both $M$ and $N$ as restrictions. We call this matroid, first defined by Brylawski [\ref{bry}], the \emph{modular sum} of $M$ and $N$, although many authors refer to the \emph{generalised parallel connection}.
		
A \emph{computable function} is one that can be calculated by a Turing machine that halts in finitely many steps; all functions defined in this paper are trivially computable. We also require that functions associated with the two theorems below, proved respectively in [\ref{gkpaper}] and [\ref{dhjpaper}], are computable. 

\begin{theorem}\label{gk}
	There is a computable function $f_{\ref{gk}} \colon \bZ^3 \to \bZ$ so that for every prime power $q$ and all integers $\ell,m \ge 2$, if $M$ is a matroid with no $U_{2,\ell+2}$-minor satisfying $\elem(M) \ge f_{\ref{gk}}(q,\ell,m) q^{r(M)}$, then $M$ has a $\PG(m-1,q')$-minor for some $q' > q$. 
\end{theorem}

\begin{theorem}\label{dhj}
	There is a computable function $f_{\ref{dhj}}\colon \bZ^3 \times \bQ \to \bZ$ so that for every prime power $q$, all $\beta \in \bQ_{>0}$ and all integers $\ell,m \ge 2$, if $M$ is a matroid with no $U_{2,\ell+2}$-minor satisfying $r(M) \ge f_{\ref{dhj}}(q,\ell,m,\beta)$ and $\elem(M) \ge \beta q^{r(M)}$, then $M$ has an $\AG(m-1,q)$-restriction or a $\PG(m-1,q')$-minor for some $q' > q$.
\end{theorem}

	Computability of $f_{\ref{gk}}$ follows from computability of the functions defined in ([\ref{gkpaper}]: Lemmas 3.3, 4.3, 5.1 and 5.3). Computability of $f_{\ref{dhj}}$ relies on computability of the functions in ([\ref{dhjpaper}]: Lemmas 4.3 and 5.2, Theorem 6.1) which themselves rely on computability of $f_{\ref{gk}}$, as well as that of the function defined in ([\ref{oldgrfpaper}], Lemma 8.1) and the main result of [\ref{polymath}], the Density Hales-Jewett theorem. Checking the computability of all these functions directly is straightforward except for the theorem in [\ref{polymath}], and in that theorem the authors take care to provide computable upper bounds for the associated function (see [\ref{polymath}], Theorem 1.5).

	\section{Geometries and Projections}
	
	In this section we discuss the matroids that we show are the densest in classes in $\bE_q$, namely the projections of projective geometries. 
	
	For each prime power $q$ and every integer $k$, let $\cP_q(k)$ denote the class of matroids of rank at least $2$ that are, up to simplification, a loopless $k$-element projection of a projective geometry. $\cP_q(0)$ is just the class of matroids that simplify to $\PG(n,q)$ for some $n \ge 1$. In general, each rank-$r$ matroid $M \in \cP_q(k)$ satisfies $\si(M) \cong \si(\wh M \con K)$, where $K$ is a rank-$k$ independent \textbf{flat} of a rank-$(r+k)$ matroid $\wh M$ such that $\wh M \del K \cong \PG(r+k-1,q)$.	($K$ is a flat because $\wh M \con K$ is loopless.)
	
	We will establish basic properties of matroids in $\cP_q(k)$ regarding their density and local structure, then give a method of recognition. 
		
	\subsection{Density}
	
	We now calculate the density of the matroids in $\cP_q(k)$. To obtain a lower bound, we first show that a projective geometry cannot be nontrivially partitioned into a small number of flats:
	
	\begin{lemma}\label{flatpartition}
		If $G \cong \PG(n-1,q)$ and $\cF$ is a partition of $E(G)$ into flats of $G$ with $|\cF| > 1$, then $|\cF| > q^{n/2-1}$. 
	\end{lemma}
	\begin{proof}
		We have $r_G(F) \le n-1$ for all $n \in \cF$, so $|G \del F| \ge |\AG(n-1,q)| = q^{n-1}$ for each $F \in \cF$. Since any two flats of $G$ with rank greater than $n/2$ intersect, there is at most one such flat $F_0 \in \cF$; if there is no such $F_0$, let $F_0 \in \cF$ be arbitrary. Now $\cF - \{F_0\}$ is a partition of $E(G \del F_0)$ into flats of rank at most $n/2$. This gives 
		\[|\cF - \{F_0\}| \ge \left(\tfrac{q^{n/2-1}-1}{q-1}\right)^{-1}|G \del F_0| > q^{-n/2}q^{n-1} = q^{n/2-1},\]
		giving the result.
	\end{proof}
	
	\begin{lemma}\label{projectiondensity}
	If $M$ is a rank-$r$ matroid in $\cP_{q}(k)$, then
	\begin{enumerate}
		\item\label{density}  there exists $d \in \left\{0, \dotsc, \tfrac{q^{2k}-1}{q^2-1}\right\}$ such that $\elem(M) = \tfrac{q^{r+k}-1}{q-1}-qd$,
		\item\label{restr} $M$ has a spanning restriction in $\cP_q(k')$ for each $k' \le k$, and
		\item\label{dumbbound} $\elem(M) \ge q^{k/2}$.
	\end{enumerate}
	\end{lemma}
	\begin{proof}
		We may assume that $M$ is (without simplification) a loopless $k$-element projection of $\PG(r+k-1,q)$. Let $\wh M$ be a rank-$(r+k)$ matroid and $K$ be a $k$-element independent flat of $\wh M$ such that $\wh M \del K \cong \PG(r+k-1,q)$ and $\wh M \con K = M$. Let $G = \wh M \del K$.
		
		 If $k' \le k$ and $F$ is a rank-$(r+k')$ flat of $G$ with $\sqcap_{\wh M}(F,K) = k'$ (which is easily obtained by taking any rank-$(r+k')$ flat containing a rank-$r$ flat that is skew to $K$), then $K$ contains $(k-k')$ coloops of $\wh M|(F \cup K)$ and it is easy to see that $M|F$ is a rank-$r$ matroid in $\cP_q(k')$, giving (\ref{restr}). 
		
		Note that every point of $M$ is a parallel class. The points of $M$ are a partition of $E(G)$ into $\elem(M)$ flats of $G$, and $\elem(M) > 1$, so $\elem(M) \ge q^{(r+k)/2-1} \ge q^{k/2}$ by Lemma~\ref{flatpartition}, giving (\ref{dumbbound}). Let $\cP$ be the collection of points of $M$ containing more than one element of $E(M)$. We have $\elem(G) - \elem(M) = \sum_{P \in \cP}(|P|-1) \equiv 0 \pmod{q}$.  For each $P \in \cP$, let $d_P = r_{\wh M}(P) \ge 2$ and let $d_{\max} = \max_{P \in \cP} d_P$. Let $F$ be the flat of $M$ spanned by $\cup \cP$, and let $\cP_0 \subseteq \cP$ be minimal so that $\cup \cP_0$ spans $F$ in $M$ (note that $|\cP_0| = r_M(F)$). We may choose $\cP_0$ so that $d_P = d_{\max}$ for some $P \in \cP_0$. Observe that $\elem(M|F) \ge \left({\tfrac{q^{d_{\max}}-1}{q-1}}\right)^{-1}\elem(\wh M|F)$. Now $\cP_0$ is a skew set of points of $M$, every pair of flats of $G$ is modular, and $K$ is a flat of $\wh M$; it follows that $\cP_0$ is a skew set of flats of $G$ whose union spans $F$ in $G$. Therefore $r_{G}(F) = \sum_{P  \in \cP_0} d_P$, so \[r_M(F) = |\cP_0| \le \sum_{P \in \cP_0} (d_P-1) = r_{G}(F) -r_M(F) \le k.\]
		If $r_M(F) = k$, then equality holds above, so $r_{G}(F) = 2k$ and $d_P = 2$ for all $P \in \cP_0$ and thus (by choice of $\cP_0$) for all $P \in \cP$. This gives $|F| = \frac{q^{2k}-1}{q-1}$ and $|\cP| \le (q+1)^{-1}|F|$, so 
		\[0 \le \elem(G) - \elem(M) \le \tfrac{q}{q+1}|F| = q\tfrac{q^{2k}-1}{q^2-1}.\]
		If $r_M(F) < k$, then $r_{G}(F) < 2k$, and 
		\[0 \le \elem(G) - \elem(M) \le |F| \le \tfrac{q^{2k-1}-1}{q-1} < q\tfrac{q^{2k}-1}{q^2-1}.\]
		 Since $\elem(G) = \tfrac{q^{r+k}-1}{q-1}$ and $\elem(G) \equiv \elem(M)  \pmod{q}$, we have (\ref{density}).
		\end{proof}
		
	We now wish to show that, given a large matroid in $\cP_q(k)$, all but a few single-element contractions give another matroid in $\cP_q(k)$ with the same density. To do this, we first need a Ramsey-type lemma about large collections of flats in a projective geometry:
		
	\begin{lemma}\label{ramseyflats}
		Let $q$ be a prime power and $n \ge 1$, $t \ge 2$ and $s \ge 0$ be integers. If $\cF$ is a set of rank-$s$ flats of $G \cong \PG(n-1,q)$ such that $|\cF| \ge q^{ts^3}$, then there is a $t$-element set $\cF_0 \subseteq \cF$ and a flat $F_0$ of $G$ such that $F_0 = \cap_{F \in \cF_0} F$ and $\{F - F_0\colon F \in \cF_0\}$ is skew in $G \con F_0$. 
	\end{lemma}
	\begin{proof}
		If $s = 0$, then the result holds vacuously since $|\cF| \le 1$. Suppose that $s > 0$ and the lemma holds for smaller $s$. Let $\cF'$ be a maximal skew subset of $\cF$. If $|\cF'| \ge s$ then the lemma holds, so we may assume that $|\cF'| < t$. Let $H = \cl_{G} (\cup \cF')$. So $r_G(H) < ts$ and $|H| < q^{ts}$; note that the number of nonempty flats of $G|H$ of rank at most $k$ is less than $\sum_{i = 1}^s |H|^i < q^{ts(s+1)}$. 
		
		By maximality of $\cF'$, each $F \in \cF$ intersects $H$ in such a flat, so there a nonempty flat $H_0$ of $G|H$ and a set $\cF'' \subseteq \cF$ such that $|\cF''| \ge q^{-ts(s+1)}|\cF| $ and $F \cap H = H_0$ for each $F \in \cF''$. Let $j = r_G(H_0)$. Now $\{F - H_0\colon F \in \cF''\}$ is a collection of rank-$(s-j)$ flats in $G \con H_0$ of size $|\cF''| \ge q^{ts^3-ts(s+1)} > q^{t(s-j)^3}$. Since $\si(G \con H_0)$ is a projective geometry over $\GF(q)$, the lemma holds by an inductive argument. 
	\end{proof}
	
	Now we argue that at most $q^{58k^4}$ points alter the density of a matroid in $\cP_q(k)$ when contracted: 
	
	\begin{lemma}\label{genericpoint}
		Let $q$ be a prime power and $k \ge 0$ be an integer. If $M \in \cP_q(k)$, and $d$ is the integer such that $\elem(M) = \frac{q^{r(M)+k}-1}{q-1} - qd$, then there is a set $X \subseteq E(M)$ such that $\elem(M|X) < q^{58k^4}$ and $\elem(M \con e) = \frac{q^{r(M \con e)+k}-1}{q-1} - qd$ for all $e \in E(M) - X$.
	\end{lemma}
	\begin{proof}
		Let $\wh M$ be a matroid having a $k$-element independent set $K$ such that $\wh M \del K \cong \PG(r(M)+k-1,q)$, and $\si(\wh M \con K) \cong \si(M)$. Let $G = \wh M \del K$. Since $M \in \cP_q(k)$, the set $K$ spans no element of $G$ in $\wh M$.  We may assume that $k \ge 1$, since $\cP_0(q)$ is the class of $\GF(q)$-representable matroids and both statements are easy in this case. 

	 Let $\cF_{-1} = \varnothing$, and for each $j \in \{0,1,2\}$, let $\cF_j$ be the set comprising every flat $F$ of $G$ so that $\sqcap_{\wh M}(F,K) = r_{G}(F) - j$ and $\sqcap_{\wh M}(F',K) \le r_G(F')-j$ for each flat $F'$ contained in $F$. Since $K$ spans no element of $G$, we have $\cF_0 = \varnothing$. The flats in $\cF_1$ correspond exactly to the points of $\wh{M} \con K$ containing more than one element of $\wh{M}$, so we have $qd = \frac{q^{r(M)+k}-1}{q-1} - \elem(M) = \sum_{F \in \cF_1} (|F|-1)$. For each $e \in E(\wh M)$, if $e$ lies in no flat in $\cF_1$ or $\cF_2$, then the flats of rank at least $2$ in  $G \con e$ that are contracted onto points of $(\wh M \con K) \con e$ correspond exactly to the flats in $\cF_1$ (contracting $e$ creates no new flat of this type, and destroys no flat of this type), so we have $qd = \frac{q^{r(M \con e)+k}-1}{q-1} - \elem(M \con e)$. Let $X = \cup(\cF_1 \cup \cF_2)$; it suffices to show that $\elem(M|X) < q^{58k^4}$. 
		
		Clearly every flat in $\cF_1 \cup \cF_2$ has rank at least $2$ and at most $k+1$. If $|\cF_1 \cup \cF_2| \ge 2kq^{(k+1)(k+2)^3}$, then by Lemma~\ref{ramseyflats} there is some $j \in \{1,2\}$, some $t \in \{2, \dotsc, k+1\}$, and a $(k+1)$-element set $\cH \subseteq \cF_j$ of rank-$t$ flats of $G$ such that $\{H - H_0\colon H \in \cH\}$ is a skew set in $G \con H_0$, where $H_0 = \cap \cH$. Note that $r_G(H_0) \le s-1$. If $j = 1$ then, since $K$ spans no element of $H_0$ we have $\sqcap_G(K,H_0) \le s-2$. If $j = 2$ then, since $H_0 \subseteq H \in \cF_2$ for each $H \in \cH_0$, we have $\sqcap_G(K,H_0) \le r_G(H_0) - 2 = s-3$. In either case, $\sqcap_G(K,H_0) < \sqcap_G(K,H)$ for each $H \in \cH$. It follows that $\{H - H_0\colon H \in \cH\}$ is a skew set of $k+1$ flats of $G \con H_0$, none of which is skew to $K$ in $G \con H_0$. Since $r_{G \con H_0}(K) \le k$, this is a contradiction. Therefore $|\cF_1 \cup \cF_2| \le 2kq^{(k+1)(k+2)^3}$. Since every flat in $\cF_1 \cup \cF_2$ has rank at most $k+1$, we have 
		\[\elem(M|X) \le \elem(G|X) \le 2kq^{k+1}q^{(k+1)(k+2)^3} < q^{4k^4}q^{(2k)(3k)^3} = q^{58k^4},\]
		as required.
	\end{proof}

		

	\subsection{Local Representability}
	
	There are many different $k$-element projections of a projective geometry over $\GF(q)$; some contain small restrictions that are not $\GF(q)$-representable (for example, the principal truncation of a plane of $\PG(n,q)$ is in $\cP_q(1)$ and has a $U_{2,q^2+q+1}$-restriction) and some do not (for example, the truncation of $\PG(n,q)$ is in $\cP_q(1)$ but contains no non-$\GF(q)$-representable flat of rank less than $n$). We define a parameter measuring the degree of `local representability' of a matroid in $\cP_q(k)$. For an integer $h \ge 1$, let $\cP_q(k,h)$ be the class of matroids in $\cP_q(k)$ having the property that every restriction of rank at most $h$ is $\GF(q)$-representable. 
	
	This property can be easily described in terms of the projection itself:
	
	\begin{lemma}\label{smoothnessequiv}
		Let $h \ge 2$ be an integer, and let $M \in \cP_q(k)$. Let $\si(M) \cong \si(\wh M \con K)$ for some rank-$(r+k)$ matroid $\wh M$ and $k$-element independent set $K$ of $\wh M$ so that $G = \wh M \del K \cong \PG(r+k-1,q)$. Then $M \in \cP_q(k,h)$ if and only if $\sqcap_{\wh M}(K,F) = 0$ for every flat $F$ of $G$ of rank at most $h+1$. 
	\end{lemma}
	\begin{proof}
		Suppose that $\sqcap_{\wh M}(K,F) > 0$ for some flat $F$ of $G$ with $r_G(W) \le h+1$. Choose $F$ to be minimal with this property, so $\sqcap_{\wh M}(K,F) = 1 < r_G(F)$. Let $F' = F$ if $r_G(F) > 2$, and if $r_G(F) = 2$, let $F'$ be a plane of $G$ containing $F$ and so that $\sqcap_{\wh M}(F',K) = 1$; such a plane exists because $r(M) > 1$ so $r_{\wh M}(K) \le r+k-2$.  By modularity of the flats in $G$ and the fact that $K$ spans no point of $F'$, there is at most one line of $G|F'$ that is not skew to $K$ in $G$; it follows that 
		\[\elem(M|F') \ge \elem(G|F') - q \ge \tfrac{q^{r_G(F')}-1}{q-1} - q > \tfrac{q^{r_G(F')-1}-1}{q-1} \ge \tfrac{q^{r_M(F')}-1}{q-1},\]
		so $M|F'$ is too dense to be $\GF(q)$-representable. We have $r_M(F') \le \max(h,3-1) \le h$, so $M \notin \cP_q(k)$. 
				
		Conversely, suppose that $\sqcap_{\wh M}(K,F) = 0$ for every flat $F$ of rank at most $h+1$ in $G$. For every independent set $I$ of $M$ with $r_M(I) \le h$, we have $\sqcap_{\wh M}(I \cup \{e\},K) = 0$ for all $e \in E(G)$, and so $e \notin \cl_{\wh M \con K}(I)$ for all $e \in E(G) - \cl_G(I)$. Therefore $M|\cl_M(I) = G|\cl_G(I)$ and thus $M|\cl_M(I) \cong \PG(r_M(I)-1,q)$, giving $M \in \cP_q(k,h)$.	
	\end{proof}
	
	In particular, it follows from the above lemma that if $M \in \cP_q(k,2)$, then $\elem(M) = \frac{q^{r(M)+k}-1}{q-1}$, since $K$ is skew to every line of $G$ and its contraction thus identifies no pair of points. 
	
	We now show that given a set $X$ of at least half the elements a very large matroid in $\cP_q(k,h)$, we can find a large contraction-minor in $\cP_q(k,h)$ that is `covered' by $X$. 
		
	\begin{lemma}\label{getgoodpg}
		There is a computable function $f_{\ref{getgoodpg}}\colon \bZ^4 \to \bZ$ so that, for every prime power $q$ and all integers $k,m \ge 0$ and $h \ge 2$, if $M \in \cP_q(k,h)$ is simple with $r(M) \ge f_{\ref{getgoodpg}}(q,k,h,m)$, and $X \subseteq E(M)$ satisfies $|X| \ge \tfrac{1}{2}|M|$, then there is a set $C \subseteq E(M)$ such that $M \con C \in \cP_q(k,h)$, $r(M \con C) \ge m$, and each parallel class of $M \con C$ intersects $X$. 
	\end{lemma}
	\begin{proof}
		Let $q$ be a prime power and $k,h,m \ge 0$ be integers. Let $m_0 = \max\left(m,(k+1)(h+2)^3+h+2\right)$. Set $f_{\ref{getgoodpg}}(q,k,h,m) = f_{\ref{dhj}}(q,q,(2q)^{-1},m_0)$.
		
		Let $M \in \cP_q(k,h)$ be simple with $r = r(M) \ge f_{\ref{getgoodpg}}(q,k,h,m)$, and let $X \subseteq E(M)$ satisfy $|X| \ge \tfrac{1}{2}|M|$. Let $\wh M$ be a matroid such that $G = \wh M \del K \cong \PG(r+k-1,q)$ and $\wh M \con K = M$ for some $k$-element independent set $K$ of $M$ that is skew in $M$ to every flat of rank at most $h+1$ in $G$. We have $|X| \ge \tfrac{1}{2}|M| \ge \tfrac{1}{2}q^{r(\wh M)-1}$. Since $G$ has no $U_{2,q+2}$-minor, Theorem~\ref{dhj} implies that $G|X$ has an $\AG(m_0,q)$-restriction $R$. 
		
		Let $C_0 \subseteq E(G)$ be maximal so that
		\begin{enumerate}[(i)]
			\item $R$ is a restriction of $G \con C_0$, and
			\item\label{skewness} $K$ is skew to in $\wh M \con C_0$ to every set of rank at most $h+1$ in $G \con C_0$. 
		\end{enumerate}
		
		 Let $J$ be the set of elements $e \in E(G \con C_0)$ such that $C_0 \cup \{e\}$ does not satisfy (\ref{skewness}). 
		\begin{claim}
		 $\elem(M \con C_0 |J) < q^{m_0}$.
		\end{claim}
		\begin{proof}[Proof of claim:]
			Suppose otherwise. Let $G' \cong \PG(r(G')-1,q)$ be a simplification of $G \con C_0$, and let $J' = J \cap E(G')$. For each $e \in J'$ there is some rank-$(h+1)$ flat $H_e$ of $G' \con e$ such that $\sqcap_{G' \con e}(H_e,K) > 0$, and so $F_e = H_e \cup \{e\}$ is a rank-$(h+2)$ flat of $G'$ with $\sqcap_{G'}(F_e,K) > 0$. Let $\cF = \{F_e\colon e \in J'\}$.
			
			Since each flat in $\cF$ has size less than $q^{h+2}$, we have $|\cF| > q^{-h-2}|J'| \ge q^{m_0-h-2} \ge q^{(k+1)(h+2)^3}$. By Lemma~\ref{ramseyflats}, there is a $(k+1)$-element set $\cF' \subseteq \cF$ and a flat $F_0$ of $G'$ such that $F_0 = \cap \cF$ and the flats $\{F - F_0\colon f \in \cF'\}$ are skew in $G' \con F_0$. Since $r_N(F_0) \le h$ we have $\sqcap_{G'}(F_0,K) = 0$, so $\{F-F_0\colon f \in \cF'\}$ is a skew collection of $k+1$ sets in $G'$, none of which is skew to $K$ in $G' \con F_0$. This contradicts the fact that $r_{G' \con F_0}(K) = r_{G'}(K) = k$. 
		\end{proof}
		
		If $R$ is not spanning in $\wh M \con C_0$, then $E(R)$ spans a proper subflat of the projective geometry $G \con C_0$, so there are at least $q^{m_0+1} > \elem(\wh M \con C_0|J)$ points of $G \con C_0$ not spanned by $E(R)$; contracting any such point not in $J$ gives a contradiction to the maximality of $C$. Therefore $R$ is spanning in $\wh M \con C_0$. Similarly, $|R| = q^{m_0} > \elem((\wh M \con C_0)|J)$, so there is some $e \in E(R) - J$; let $C = C_0 \cup \{e\}$. Now $E(R)$ contains a simplification $R'$ of $G \con C$, and $K$ is a rank-$k$ flat of $\wh M \con C$ skew to every flat of $R'$ of rank at most $h+1$; it follows that $M \con C = (\wh M \con (C \cup K)) \in \cP_q(k,h)$. Since $E(R') \subseteq X$ contains a simplification of $G \con C$, the lemma follows. 
	\end{proof}
			
	\subsection{Recognition}

We now prove a result that will identify matroids in $\cP_q(k)$. We use the fact (see [\ref{oxley}], Proposition 7.3.6) that, if $M$ and $N$ are matroids on a common ground set $E$, then $N$ is a projection of $M$ if and only if $\cl_M(X) \subseteq \cl_N(X)$ for all $X \subseteq E$. Since we are concerned with isomorphism and simplification as well, it is convenient to give a slightly different version of this statement. If $M$ and $N$ are matroids then we say that  $\varphi\colon E(M) \to E(N)$ is a $\emph{projective map}$ from $M$ to $N$ if $\varphi(\cl_M(X)) \subseteq \cl_N(\varphi(X))$ for all $X \subseteq E(M)$. It is routine to show that, if $\varphi$ is a projective map from $M$ to $N$, then the matroid $N|\varphi(E(M))$ is isomorphic to the simplification of a projection of $M$; the set of elements of $M$ that map to a given element corresponds to a parallel class in this projection of $M$.

 We use this to identify projections of projective geometries, as well as projections of representable matroids that are the union of at most two flats of a projective geometry. If $G \cong \PG(n-1,q)$ in the following lemma, then the conclusion gives $M \in \cP_q(k)$ for some $k$. 

\begin{lemma}\label{recognisepqk}
	Let $M$ be a simple matroid and let $G$ be a restriction of $\PG(n-1,q)$ such that $E(G)$ is the union of at most two flats of $\PG(n-1,q)$. If there is a surjective map $\varphi \colon E(G) \to E(M)$ such that, for every triangle $T$ of $G$, either $|\varphi(T)| = 1$ or $\varphi(T)$ is a triangle of $M$, then $M$ is, up to simplification, a projection of $G$. 
\end{lemma}
\begin{proof}
	It suffices to show that $\varphi(\cl_G(X)) \subseteq \cl_M(\varphi(X))$ for all $X \subseteq E(G)$. Let $X \subseteq E(G)$ and let $F_1,F_2$ be flats of $\PG(n-1,q)$ so that $E(G) = F_1 \cup F_2$. Let $X_i = X \cap F_i$. It is easy to see that every element of $\cl_G(X_1 \cup X_2)$ is either in $\cl_G(X_1) \cup \cl_G(X_2)$, or lies in a triangle of $G$ containing a point of $\cl_G(X_1)$ and a point of $\cl_G(X_1)$. It follows routinely that there is a sequence of sets 
	\[X_1 \cup X_2 = Y_0,Y_1,Y_2, \dotsc, Y_m = \cl_G(X_1 \cup X_2)\] so that, for each $i \in \{1, \dotsc, m\}$, we have $Y_i = Y_{i-1} \cup \{e_i\}$ for some $e_i$ that is contained in a triangle of $G|Y_i$. An easy inductive argument implies that $\varphi(Y_i) \subseteq \cl_M(\varphi(X_1 \cup X_2))$ for each $i$. Therefore $\varphi(\cl_G(X_1 \cup X_2)) \subseteq \cl_M(\varphi(X_1 \cup X_2))$, as required. 
	\end{proof}
	
	We now prove the lemma allowing us to recognise general matroids in $\cP_q(k)$ from a modularity assumption.

\begin{lemma}\label{recognition}
	Let $q$ be a prime power, and $\ell \ge 2$ and $j,t \ge 0$ be integers. Let $s = 10\ell+t$. If $M$ is a matroid of rank at least $2s+t$ with no $U_{2,\ell}$-restriction, and $K$ is a rank-$t$ subset of $E(M)$ such that $M \con K \in \cP_q(j,s)$, and for every rank-$(t+1)$ flat of $M$ containing $K$ and every line $L$ of $M$, the pair $(L,K)$ is modular, then there exists $D \subseteq \cl_M(K)$ such that $M \del D \in \cP_q(k)$ for some $k$. 
\end{lemma}
\begin{proof}
	We may assume that $M$ is simple and that $K$ is a flat of $M$. 
	
	\begin{claim}
		Every flat $F$ of $M$ with $r_M(G) \le 2s$ and $\sqcap_M(F,K) = 0$ satisfies $M|F \cong \PG(r_M(F)-1,q)$.
	\end{claim}
	\begin{proof}[Proof of claim:]
		Let $r =r_M(F) \le 2s$. Now $M|F$ is a rank-$r$ restriction of $(M\con K)|\cl_{M \con K}(F)$, which is a rank-$r$ flat of a matroid in $\cP_q(j,s)$ so is isomorphic to $\PG(r-1,q)$. Moreover, for all distinct $e,f \in F$, the line $\cl_M(\{e,f\})$ intersects $\cl_M(K \cup \{x\})$ for all $x \in \cl_{M \con K}(\{e,f\})$, as otherwise this pair of flats fails to be modular. Therefore $M|F$ is a simple rank-$r$ restriction of $\PG(r-1,q)$ in which every line contains at least $q+1$ points; it follows that $M|F \cong \PG(r-1,q)$.
		\end{proof}
	
	Let $F_0$ be a rank-$2s$ flat of $M$ that is skew to $K$, and let $X_0 \subseteq E(M)$ be maximal such that $F_0 \subseteq X_0$ and $M|X_0 \in \cP_q(k)$ for some $k$. Let $G_0 \cong \PG(n-1,q)$, and let $\varphi_0\colon E(G_0) \to X_0$ be the surjective map associated with this projection. Let $G \cong \PG(n,q)$ be an extension of $G_0$. If $E(M) = X_0 \cup K$, then the lemma clearly holds, so we may assume that there is some $e \in E(M) - (X_0 \cup K)$.  
	
	 \begin{claim}
	 	There is a rank-$s$ flat $\wh{X}$ of $M$, containing $e$ and skew to $K$, so that $M|(\wh{X} \cap X_0) \cong \PG(s-2,q)$ and $M|\wh{X} \cong \PG(s-1,q)$. 
	 \end{claim}
	 \begin{proof}[Proof of claim:]
	 	Note that $r_M(X_0) \ge 2s \ge s + t$; let $Y$ be a rank-$(s-1)$ flat of $M|X_0$ that is skew to $K$ in $M \con e$. Since $M|X_0 \in \cP_q(k)$, the matroid $M|Y$ has a $\PG(s-2,q)$-restriction. However, $M|Y$ is contained in a $\PG(s-2,q)$-restriction of $M$ by the first claim, so $M|Y \cong \PG(s-2,q)$ and $Y$ is a flat of $M$. Let $\wh X = \cl_M(Y \cup \{e\})$. Now $\wh X$ is skew to $K$ in $M$, and by the first claim, $M|\wh{X} \cong \PG(s-1,q)$, and $Y$ is a hyperplane of $M|\wh X$. If $f \in (\wh{X} - Y) \cap X_0$, then $\cl_M(\{e,f\})$ meets $Y$ in some $f_0$, so $e \in \cl_M(\{f,f_0\}) - X_0$. Since $\cl_M(\{f,f_0\})$ is a line of $X_0$, it contains $q+1$ points of $X_0$, so contains at least $q+2$ points of $M$. Since $\cl_M(\{f,f_0\})$ is skew to $K$, this contradicts the first claim. Therefore $Y = \wh X \cap X_0$, and the claim follows.
	 \end{proof}
	
	Let $\wh{G}$ be a $\PG(s-1,q)$-restriction of $G$ so that $G|(E(\wh{G}) \cap E(G_0)) \cong \PG(s-2,q)$. Let $\whp\colon E(\wh{G}) \to \wh X$ be a matroid isomorphism between $\wh{G}$ and $M|\wh{X}$ so that $\whp(x) = \varphi(x)$ for all $x \in E(\wh{G}) \cap E(G_0)$. Let $a = \whp^{-1}(e)$, noting that $a \notin E(G_0)$. 
	
	\begin{claim}\label{jdef}
		For each $b \in E(G)-\{a\}$, there is some $x(b) \in E(M)$ and a $10$-element independent set $J(b)$ of $\wh{G} \del a$ such that
		\begin{enumerate}[(i)]
		\item\label{jdefi} $x(b) \in \cl_M(\{e,\varphi_0(b_0)\})$ where $\{b_0\} = \cl_G(\{a,b_0\}) \cap E(G_0)$, and $x(b) = \varphi_0(b_0)$ if and only if $b = b_0$, and 
		\item\label{jdefii}  for all $c \in J(b)$, there is a triangle $\{c,d_0,b\}$ of $G$ such that $d_0 \in E(G_0)$ and $\{\whp(c),\varphi_0(d_0),x(b)\}$ is a triangle of $M$, and
		\item\label{jdefiii} if $b \in E(\wh{G})$ then $x(b) = \whp(b)$, and if $b \in E(G_0)$ then $x(b) = \varphi_0(b)$.
		\end{enumerate}
	\end{claim}
	\begin{proof}[Proof of claim:]
		If $b \in E(G_0)$ then set $x(b) = \varphi_0(b)$ and set $J(b)$ to be an arbitrary $10$-element independent set of $G_0 \del b$; the conditions follow easily for $x(b)$ and $J(b)$ since $b \notin J(b)$ and $\varphi_0$ is a projective map. If $b \in E(\wh G) - E(G_0)$, then set $x(b) = \whp(b)$ and $J(b)$ to be an arbitrary $10$-element independent set of $\wh G \del (\cl_G(\{a,b\}) \cup (E(G_0) \cap E(\wh G)))$; again, the conditions easily follow from the fact that $\whp$ is an isomorphism and $E(G_0) \cap E(\wh G)$ is a hyperplane of $\wh G$. 
		
		Suppose that $b \in E(G) - (E(G_0) \cup E(\wh G))$. Let $\{b_0\} = \cl_G(\{a,b\}) \cup E(G_0)$. Note that $r(\wh G \del (E(G_0) \cap E(\wh G)) = s \ge 10\ell+t \ge 9\ell+t+3$; let $J$ be a $(9\ell+t+2)$-element independent set of $\wh G \con a$ that is disjoint to $E(G_0)$. we will choose $J(b)$ to be an appropriate subset of $J$.  Let $c \in J$, and let $\{c_0\} = \cl_G(\{a,c\}) \cap E(G_0)$. The set $\{c_0\colon c \in J\}$ is a $|J|$-element independent set of $\wh{G}$, so there is at most one $c \in J$ for which $\varphi_0(c_0) = \varphi_0(b_0)$; let $J' \subseteq J$ be a $(9\ell+t+1)$-element set containing no such $c$. Let $c \in J'$, and $\{d_0\} = \cl_G(\{c,b\}) \cap E(G_0)$. Since $\whp$ is an isomorphism, the set $\whp(\{c_0,c,a\})$ is a triangle of $M$. Since $\varphi_0$ is projective and $b_0 \ne c_0$, the set $\varphi_0(\{c_0,d_0,b_0\})$ is also a triangle of $M$.  Now $e \notin \cl_M(\varphi_0(G_0))$, so the latter triangle does not span $e$ in $M$; hence, $\cl_M(\{\whp(c),\varphi_0(d_0)\})$ and $\cl_M(\{e,\varphi_0(b_0)\})$ are distinct coplanar lines of $M$. 
		
		If $\{e,\varphi_0(b_0)\}$ is not skew to $K$ in $M$, then it is contained in a rank-$(t+1)$ flat of $M$ containing $K$, and it follows from the hypothesis that for each $c \in J'$, the coplanar lines $\cl_M(\{e,\varphi_0(b_0)\})$ and $\cl_M(\{\whp(c),\varphi_0(d_0)\})$ form a modular pair, and so intersect at some $x_c \in E(M)$. If $\{e,\varphi_0(b_0)\}$ is skew to $K$ in $M$, then since $M|\whp(J') \cong \wh G|J'$, the set $\whp(J')$ is independent in $M$, so there is a $(9\ell+1)$-element subset $J''$ of $J'$ such that the plane $\cl_M(\{e,\varphi_0(b_0),\whp(c)\})$ is skew to $K$ in $M$ for each $c \in J''$. By the first claim, the restriction of $M$ to such a plane is isomorphic to $\PG(2,q)$, and so again the lines $\cl_M(\{\whp(c),\varphi_0(d_0)\})$ and $\cl_M(\{e,\varphi_0(b_0)\})$ intersect at a point $x_c \in E(M)$. 
		
		In either case above, there is a $(9\ell+1)$-element subset $J'$ of $J$ such that the point $x_c \in \cl_M(\{e,\varphi_0(c_0)\})$ is well-defined for all $c \in J'$. This line contains at most $\ell$ elements of $M$, so there is a $10$-element set $J(b) \subseteq J''$ and some $x \in E(M)$ so that $x_c = x$ for all $c \in J(b)$. Since $\{e,\whp(c),\varphi_0(c_0)\}$ and $\{\varphi_0(c_0),\varphi_0(d_0),\varphi_0(b_0)\}$ are non-collinear triangles of $M$ and $\{\whp(c),\varphi_0(d_0),x(b)\}$ is a triangle of $M$, we have $x(b) \ne \varphi_0(b_0)$. Now (\ref{jdefi}),(\ref{jdefii}) and (\ref{jdefiii}) follow from construction. 
		\end{proof}
	
	Define $\varphi\colon E(G) \to E(M)$ by $\varphi(a) = \{e\}$ and $\varphi(b) = x(b)$ for each $b \in E(G)-a$, noting that $\varphi$ is an extension of both $\whp$ and $\varphi_0$. 
	
	\begin{claim}
		For each triangle $T$ of $G$, either $|\varphi(T)| = 1$ or  $\varphi(T)$ is a triangle of $M$. 
	\end{claim}
	\begin{proof}[Proof of claim:] This is clear if $T \subseteq E(\wh{G}) \cup E(G_0)$, since any such $T$ is a triangle of either $\wh{G}$ or $G_0$, and both $\wh{\varphi}$ and $\varphi_0$ are projective maps. It follows from this observation and Lemma~\ref{recognisepqk} that $\varphi|(E(G_0) \cup E(\wh{G}))$ is a projective map from $G|(E(G_0) \cup E(\wh{G}))$ to $M$.
	
	Let $T = \{b^1,b^2,b^3\}$ be a triangle of $G$. If $b^1 = a$, then let $\{b_0\} = E(G_0) \cap \cl_G(T)$. Note that $\varphi(b_0) \in X_0$ and $\varphi(a) = e \notin X_0$, and that $b^2$ and $b^3$ are not both $b_0$; it follows from \ref{jdef}(\ref{jdefi}) that $\varphi(T)$ is a triangle of $G$. We may thus assume that $a \notin T$. 
	
	Let $i \in \{1,2,3\}$. We have $M|\varphi(J(b^i)) \cong \wh{G } |J(b^i)$, so $\varphi(J(b^i))$ is a $10$-element independent set of $M$. We can therefore choose elements $c^i\colon i \in \{1,2,3\}$ so that $c^i \in J(b^i)$ for each $i$, and so that 
	\begin{itemize}
		\item $c^i \notin \cl_G(T \cup \{c_1,\dotsc,c_{i-1}\})$ for each $i \in \{1,2,3\}$, and
		\item $\varphi(c^i) \notin \cl_M(\varphi(T) \cup \{\varphi(c^1),\dotsc,\varphi(c^{i-1})\})$ for each $i \in \{1,2,3\}$. 
	\end{itemize}
	(This choice is possible since for each $i$ the set of invalid $c^i$ has rank at most $4$ in $G$ and the set of invalid $\varphi(c^i)$ has rank at most $5$ in $M$.)
	
	For $i \in \{1,2,3\}$, let $x^i = \varphi(b^i)$, let $y^i = \varphi(c^i)$, and let $T^i = \{c^i,d^i_0,b^i\}$ be a triangle of $G$ so that $d^i_0 \in E(G_0)$ and $\varphi(T^i)$ is a triangle of $M$. Let $z^i = \varphi(d^i_0)$ for each $i$, and let $W = \{c^1,c^2,c^3,d^1_0,d^2_0,d^3_0\}$. By construction of the $c^i$, we have $G|W \cong U_{5,6}$.
	
	If $\varphi(T)$ is not a triangle or singleton of $M$, then either $M|\varphi(T) \cong U_{3,3}$ or $M|\varphi(T) \cong U_{2,2}$. In the former case, we have $r_M(\varphi(T) \cup \{y^1,y^2,y^3\}) = 6$ so, since $\{x^i,y^i,z^i\}$ is a triangle of $M$ for each $i$, we also have $M|\varphi(W) \cong U_{6,6}$. If $M|\varphi(T) \cong U_{2,2}$, then we may assume that $x^1=x^2\ne x^3$. By choice of $y^1,y^2,y^3$, we therefore have $M|\varphi(W) \cong U_{3,4} \oplus U_{2,2}$, where $\{y^1,y^2,z^1,z^2\}$ is the four-element circuit. Neither $U_{6,6}$ nor $U_{3,4} \oplus U_{2,2}$ is projection of $M|W \cong U_{5,6}$; since $W \subseteq E(G_0) \cup E(\wh{G})$, we thus have a contradiction to the fact that $\varphi |E(G_0) \cup E(\wh{G})$ is a projective map from $G|(E(G_0) \cup E(\wh{G}))$ to $M$.  
	\end{proof}
	It follows from Lemma~\ref{recognisepqk} that $M|\varphi(E(G)) \in \cP_q(k)$ for some $k$; since $X_0 \cup \{e\} \subseteq \varphi(E(G))$, this contradicts the maximality of $X_0$. 
	\end{proof}
	


\section{Exponentially Dense Classes}

	Let $\fE_q$ denote the collection of all base-$q$ exponentially dense minor-closed classes of matroids. Each class in $\bE_q$, by definition, does not contain arbitrarily long lines or arbitrarily large projective geometries over fields larger than $\GF(q)$. Such classes are also `controlled' in other ways that we describe in the following lemma, stating them all in terms of a single parameter $c$ for convenience.  We will refer freely to the parameter $c = c_{\cM}$ for a given $\cM\in \bE_q$ throughout the paper. 
	
	\begin{lemma}\label{controlclass}
		For every $\cM \in \fM_q$ there is an integer $c = c_{\cM}$ so that 
		\begin{enumerate}
			\item\label{c1} $U_{2,c+2} \notin \cM$.
			\item\label{c2} $\PG(c-1,q') \notin \cM$ for all $q' > q$,
			\item\label{c3} $\elem(M) < q^{r(M)+c}$ for all $M \in \cM$,
			\item\label{c4} $\elem(M \con C) > q^{-c-r_M(C)}\elem(M)$ for all $M \in \cM$ and $C \subseteq E(M)$, and
			\item\label{c5} $\cM \cap \cP_q(k) = \varnothing$ for all $k \ge c$.
		\end{enumerate}
		Moreover, there is a computable function $f_{\ref{controlclass}}\colon \bZ^3 \to \bZ$ so that, if $\ell \ge 2$ and $s \ge 2$ are integers such that $U_{2,\ell+2} \notin \cM$ and $\PG(s-1,q') \notin \cM$ for all $q' > q$, then $c_{\cM} \le f_{\ref{controlclass}}(q,\ell,s)$.  
	\end{lemma}
	\begin{proof}
		Since $\cM \in \fE_q$, there are integers $\ell,s \ge 2$ such that $U_{2,\ell+2} \notin \cM$ and $\PG(s-1,q') \notin \cM$ for all $q' > q$. (This is true because $U_{2,q'+1} \cong \PG(1,q') \notin \cM$ for all $q' > \ell$, and $\cM$ does not contain all $\GF(q')$-representable matroids for any $q' \in \{q+1,\dotsc,\ell\}$). We show that $c$ can be defined to depend computably on $q$ and these two parameters; indeed, let $d = \left\lceil \log_q(f_{\ref{gk}}(q,\ell,c_2)) \right\rceil$ and set $c = \max(\ell,s,q^{2d+4})$.
		
		Clearly $c$ satisfies (\ref{c1}) and (\ref{c2}) and, by  Theorem~\ref{gk}, we have $\elem(M) \le q^{r(M)+d} < q^{r(M)+c}$ for all $M \in \cM$, so $c$ satisfies (\ref{c3}). Let  $C \subseteq E(M)$ and $\cF$ be the collection of rank-$(r_M(C)+1)$ flats of $M$ containing $C$; we have $\elem(M) < \sum_{F \in \cF} \elem(M|F) \le \elem(M \con C)q^{r_M(C)+d+1} < q^{r_M(C)+c}\elem(M \con C)$, so $c$ satisfies (\ref{c4}).  
		
		Suppose that $k \ge q^{2d+4}$ and $M \in \cM \cap \cP_q(k)$. By Lemma~\ref{projectiondensity} we have $q^{r(M)+d} > \elem(M) \ge q^{2d+4}$, so $r(M) \ge d+4$. Now $M$ has a spanning restriction in $\cP_q(d+2)$, giving 
		\[\elem(M) \ge \tfrac{q^{r(M)+d+2}-1}{q-1} - q\tfrac{q^{2(d+2)}-1}{q^2-1} > q^{r(M)+d+1}-q^{2d+4} > q^{r(M)+d},\] a contradiction. Since $c \ge q^{2d+4}$, it follows that $c$ satisfies (\ref{c5}). 
	\end{proof}
	
	\subsection{Connectivity}
	
	A rank-$r$ matroid is \emph{weakly round} if every cocircuit has rank at least $r-1$. This is a very strong connectivity property that is a slight relaxation of \emph{roundness}, a notion introduced by Kung  (where cocircuits are required to be spanning) under the name of \emph{non-splitting} in [\ref{kungroundness}]. Our first lemma shows that an exponentially dense matroid of large rank in a class in $\bE_q$ has a comparably dense restriction of large rank that is also weakly round: 
	
	\begin{lemma}\label{roundness}
		There is a computable function $f_{\ref{roundness}}\colon \bZ^4 \to \bZ$ so that, for every prime power $q$ and all integers $c,b,m \ge 0$, if $\cM \in \fE_q$ satisfies $c_{\cM} \le c$ and $g\colon \bZ \to \bZ$ is a function so that $g(n) > qg(n-1) \ge q^{b-1}$ for all $n > b$, then every $M \in \cM$ satisfying $r(M) \ge f_{\ref{roundness}}(q,c,b,m)$ and $\elem(M) \ge g(r(M))$ has a weakly round restriction $M'$ of rank at least $m$, such that either $M' = M$ or $\elem(M') > g(r(M'))$. 
	\end{lemma}
	\begin{proof}
		Let $q$ be a prime power, and $b,m,c \ge 0$ be integers. Let  $\varphi = \tfrac{1}{2}(1 + \sqrt{5})$. Let $t$ be an integer so that $(q\varphi^{-1})^t > q^{c+1}$, and set $f_{\ref{roundness}}(q,c,b,m) = n = \max(b,m+t)$. 
		
		Let $\cM \in \fE_q$ satisfy $c > c_{\cM}$. Let $M \in \cM$ satisfy $r = r(M) \ge n$ and $\elem(M) \ge g(r)$. Let $M'$ be a minimal restriction of $M$ such that $E(M')$ is  a flat of $M$ and $\elem(M') \ge \varphi^{r(M')-r}\elem(M)$. Let $r' = r(M')$. If $M'$ is not weakly round, then there are flats $F_1$ and $F_2$ of $M'$ with $F_1 \cup F_2 = E(M')$ and $r(M'|F_i) \le r' - i$ for each $i \in \{1,2\}$; by minimality this gives $\elem(M') \le \elem(M'|F_1) + \elem(M'|F_2) < (\varphi^{-1} + \varphi^{-2})\varphi^{r'-r}\elem(M) \le \elem(M')$, a contradiction. So $M'$ is weakly round. 
		
		Note that, since $r \ge b$, we have $g(r) \ge q^{r-1}$, so 
		\[q^{r'+c} \ge \elem(M') \ge (q\varphi^{-1})^{r-r'}q^{r'-r}\elem(M) \ge (q\varphi^{-1})^{r-r'}q^{r'-1}.\] Therefore $q^{-1}(q\varphi^{-1})^{r-r'} \le q^c$, giving $r - r' \le t$ and so $r' \ge n-t \ge m$.
		
		If $r' = r$, then $M' = M$ since $E(M')$ is a flat of $M$; since $r(M) \ge n \ge m$ the lemma holds. If $r' < r$, then $\elem(M') \ge \varphi^{r'-r}g(r) = (q \varphi^{-1})^{r-r'}q^{r' - r}g(r) > g(r')$, since $r-r' > 0$ and $g(r) > q^{r-r'}g(r')$ by definition of $g$; again, the lemma holds.\end{proof}
		
	Weak roundness is clearly closed under contraction, and is thus easy to exploit to contract two restrictions of different ranks together:
		
	\begin{lemma}\label{menger}
		If $M$ is a weakly round matroid and $X$ and $Y$ are subsets of $E(M)$ such that $r_M(X) < r_M(Y)$, then $M$ has a minor $N$ so that $M|X = N|X, M|Y = N|Y$ and $Y$ is spanning in $N$. 
	\end{lemma}
	\begin{proof}
		Let $N$ be a minimal minor of $M$ so that $M|X = N|X$, $M|Y = N|Y$ and $N$ is weakly round. By minimality we have $\elem(N) = \cl_N(X) \cup \cl_N(Y)$. If $Y$ is not spanning in $N$ then $r_N(Y) \le r(N)-1$ and $r_N(X) \le r(N)-2$; these facts together contradict weak roundness of $N$. 
	\end{proof}
	
	\subsection{Stacks}
	
	We will use Lemma~\ref{menger} to contract two `incompatible' restrictions together in a large matroid: an affine geometry over $\GF(q)$, and a stack. For each prime power $q$ and integers $k \ge 0$ and $t \ge 2$, a matroid $S$ is a \emph{$(q,k,t)$-stack} if there are disjoint sets $F_1, \dotsc, F_k \subseteq E(S)$ such that the union of the $F_i$ is spanning in $S$ and, for each $i \in \{1, \dotsc, k\}$ the matroid $(S \con (F_1 \cup \dotsc \cup F_{i-1}))|F_i$ has rank at most $t$ and is not $\GF(q)$-representable. 
	
	Note that a $(q,k,t)$-stack has rank at least $2k$ and at most $tk$. Stacks are far from being $\GF(q)$-representable; our first lemma, proved in ([\ref{dhjpaper}], Lemma 4.2), shows that a stack of height $\binom{k+1}{2}$ on top of a large projective geometry guarantees a rank-$k$ flat disjoint from the geometry. 
	
	\begin{lemma}\label{stackfindprojection}
		Let $q$ be a prime power and $h$ and $t$ be nonnegative integers. If $M$ is a matroid with a $\PG(r(M)-1,q)$-restriction $R$ and a $(q,\binom{k+1}{2},t)$-stack restriction $S$, then $E(S) - E(R)$ contains a rank-$k$ flat of $M$. 
	\end{lemma}
	
	Such a flat gives rise to a $k$-element projection of $R$. Since matroids in a given class $\cM \in \bE_q$ do not contain these projections for arbitrarily large $k$, the above implies that a large stack and a large affine geometry restriction cannot coexist in a weakly round matroid in $\cM$:
	
	\begin{lemma}\label{stackwin}
	 Let $\cM \in \bE_q$, let $c = c_{\cM}$, and let $s \ge 2$ be an integer. Then $\cM$ contains no weakly round matroid with an $\AG(c^2s,q)$-restriction and a $(q,c^2,s)$-stack restriction.
	\end{lemma}
	\begin{proof} Let $c = c_{\cM}$ and let $M \in \cM$ be weakly round with an $\AG(c^2s,q)$-restriction $R$ and a $(q,c^2,s)$-stack restriction $S$. Since $r(S) \le c^2 s = r(R)-1$, Lemma~\ref{menger} gives a minor $N$ of $M$ with $S$ as a restriction and $R$ as a spanning restriction. Since $r(S) < r(R)$ there is some $e \in E(R) - \cl_N(E(S))$; now $N \con e$ has $S$ as a restriction and $(N \con e)|(E(R \con e))$ has a $\PG(r(N \con e)-1,q)$-restriction $R'$. Since $c^2 \ge \binom{c+1}{2}$, the matroid $S$ has a $(q,\binom{c+1}{2},s)$-stack restriction, so there is a rank-$c$ flat $F$ of $M$ disjoint from $E(R')$ by Lemma~\ref{stackfindprojection}; now $r(R') > c+1$; it follows that $(M \con F)|E(R') \in \cM \cap \cP_q(c)$, contradicting Lemma~\ref{controlclass}.
	\end{proof}
		
	The affine geometries in the lemma above will be obtained from Theorem~\ref{dhj}. The following is a more convenient version of the theorem that applies within a particular class in $\bE_q$; the equivalence, with $f_{\ref{mccdhj}}(q,c,m,\beta) = f_{\ref{dhj}}(q,c,\max(m,c+1),\beta)$, is easy to see.
	
	\begin{theorem}\label{mccdhj}
		There is a computable function $f_{\ref{mccdhj}}\colon \bZ^3 \times \bQ \to \bZ$ so that, for every prime power $q$, all integers $c,m \ge 2$ and all $\beta \in \bQ_{>0}$, if $\cM \in \fE_q$ satisfies $c_{M} \le c$, and $M \in \cM$ satisfies $r(M) \ge f_{\ref{mccdhj}}(q,c,m,\beta)$ and $\elem(M) \ge \beta q^{r(M)}$, then $M$ has an $\AG(m-1,q)$-restriction. 
	\end{theorem}
		 		
	\section{Minimality}
	
	The results in the previous section imply that a dense matroid in a class $\cM \in \bE_q$ has a dense weakly round restriction that itself has a large affine geometry restriction. The following lemma, roughly, shows that a matroid which is minor-minimal with respect to being dense and having this geometry as a restriction, has a spanning projective geometry after contracting just a few elements.  
	
	\begin{lemma}\label{getnearspanningpg}
		There is a computable function $f_{\ref{getnearspanningpg}}\colon \bZ^2 \to \bZ$ so that for every prime power $q$ and integer $c \ge 2$, if $\cM \in \fM_q$ satisfies $c_{\cM} \le c$, and $M \in \cM$ is weakly round, satisfies $\elem(M) \ge q^{r(M)-1}$, and has an $\AG(f_{\ref{getnearspanningpg}}(q,c)-1,q)$-restriction $R$ such that $\elem(M) > q \elem(M \con e)$ for all $e \in E(M) - \cl_M(E(R))$, then there is a set $C \subseteq E(M)$ such that $M \con C$ has a $\PG(r(M \con C)-1,q)$-restriction and $r_M(C) \le f_{\ref{getnearspanningpg}}(q,c).$ 
	\end{lemma}
	\begin{proof}
		Let $q$ be a prime power and $c \ge 2$ be an integer. Let $s = q^{q^{5c^2}}+4$. Let 
		\[\beta = q^{-1}, \ \ \beta_1 = \tfrac{1}{2}\beta q^{-c^2s-c}, \ \ \beta_2 =  \tfrac{1}{2}\beta_1 q^{-2q^{5c^2}}. \] 
		Let $n_0$ and $t$ be integers so that $\beta_1 q^{n_0} \ge q^{c^2s+c}$ and $\beta_1q^{n_0-c^2s-q^{5c^2}} \ge 2$, and $\beta_2 > q^{c-t}$. Finally, set 
		\[f_{\ref{getnearspanningpg}}(q,c) = \max(n_0,c^2s+1,20(c+2) +  c^2s + q^{4c^2}+5).\]

		Let $\cM \in \fE_q$ satisfy $c_{\cM} \le c$, let $n \ge f_{\ref{getnearspanningpg}}(q,c)$ be an integer, and let $M \in \cM$ be weakly round such that $\elem(M) \ge q^{r(M)-1}$ and $M$ has an $\AG(n-1,q)$-restriction $R$ so that $\elem(M) > q \elem(M \con e)$ for all $e \in E(M) - \cl_M(E(R))$. We may assume that $M$ is simple. If $r(M) < n + n_0$ then we can contract at most $n_0-1$ elements so that $R$ is a spanning restriction and contract a further element of $R$ to obtain a spanning projective geometry; since $n_0  \le f_{\ref{getnearspanningpg}}(q,c)$ the lemma is satisfied. We may therefore assume that $r(M) \ge n + n_0$.
		
		We say a line of a matroid is \emph{long} if it contains at least $q+2$ points, and \emph{short} if it contains at most $q$ points. Let $\cL_0$ be a maximal skew collection of long lines of $M$ and let $F = \cl_M(\cup \cL_0)$. Note that $M|F$ is a $(|\cL_0|,q,2)$-stack and is therefore also a $(|\cL_0|,q,s)$-stack; by Lemma~\ref{stackwin}, $M$ has no $(c^2,q,s)$-stack restriction, so $|\cL_0| < c^2$ and $r_M(F) = 2|\cL_0| < 2c^2$. By maximality of $\cL_0$, no long line of $M$ is skew to $F$. For each $e \in E(M) - F$ it follows that the number of long lines of $M$ through $e$ does not exceed $\elem((M \con e)|\cl_{M \con e}(F)) < q^{2c^2+c} \le q^{3c^2}$. We claim that every point not spanned by $F$ or $E(R)$ is on few short lines:
		
		\begin{claim}
			Each $e \in E(M) - (F \cup \cl_M(E(R)))$ is in fewer than $q^{4c^2}$ short lines of $M$. 
		\end{claim}
		\begin{proof}[Proof of claim:]
		For each such $e$, let $\cL^-$ and $\cL^+$ respectively denote the sets of short and long lines of $M$ through $e$, and let $\cL^=$ be the set of all other lines through $e$.  Note that $|\cL^+| \le q^{3c^2}$ and $\elem(M \con e) = |\cL^+| + |\cL^-| + |\cL^{=}|$. Now 
		\begin{align*}
			\elem(M) &= \sum_{L \in \cL^-}(|L|-1) + q|\cL^=| + \sum_{L \in \cL^+}(|L|-1)\\
			&\le (q-1)|\cL^-| + q|\cL^=| + (c+1)|\cL^+|\\
			&= q\elem(M \con e) + (c+1-q)|\cL^+| - |\cL^-|
		\end{align*}
		Now $\elem(M) > q\elem(M \con e)$, so $|\cL^-| < (c+1-q)|\cL^+| < cq^{3c^2} < q^{4c^2}$. 
		\end{proof}
		Recall that $M|F$ is a $(|\cL_0|,q,s)$-stack. Let $S$ be an $(j,q,s)$-stack restriction of $M$ containing $F$ for which $j$ is as large as possible. By Lemma~\ref{stackwin} we have $j <c^2$, so $r_M(S) < c^2s$. Let $M_1 = M \con E(S)$, noting that $r(M_1) > r(M)-c^2s$; by maximality of $j$ we know that every rank-$s$ restriction of $M_1$ is $\GF(q)$-representable. We also have $\elem(M_1) \ge q^{-c^2s-c}\elem(M) \ge q^{-c^2s-c}\beta q^{r(M)} \ge 2\beta_1q^{r(M_1)}$.
		
		Let $X = \cl_{M_1}(E(M_1) \cap E(R))$; note that $\elem(M_1|X) \le q^{n+c} \le \beta_1 q^{n+n_0-c^2s} \le \beta_1 q^{r(M)-c^2s} \le \beta_1q^{r(M_1)}$, so $\elem(M_1 \del X) \ge \beta_1 q^{r(M_1)}$. Moreover, every nonloop of $M_1 \del X$ is in fewer than $q^{4c^2}$ $q$-short lines of $M$ and is therefore in fewer than $q^{4c^2}$ short lines of $M_1$.
				
		\begin{claim}
			There are disjoint sets $J,Z \subseteq E(M_1)$, satisfying  $r_{M_1}(J) \le q^{4c^2}$ and $\elem((M_1 \con J)|Z) \ge \beta_2 q^{r(M_1 \con J)}$, so that every short line of $M_1 \con J$ is disjoint to $Z$. 
		\end{claim}
		\begin{proof}[Proof of claim:]
			For each nonloop $e$ of $M_1$, let $J_e$ be the closure in $M_1$ of the union of the short lines of $M_1$ containing $e$. If $e \notin X$ then there are fewer than $q^{4c^2}$ such lines, so $r(M_1|J_e) \le q^{4c^2}$ and $\elem(M_1|J_e) < q^{q^{4c^2}+c} < q^{q^{5c^2}}$. Recall that $\elem(M_1\del X) \ge \beta q^{r(M_1)}$; let $Z'$ be a set of nonloops of $M_1\del X$ satisfying $|Z'| = \lfloor \beta_1 q^{-q^{5c^2}} q^{r(M_1)}\rfloor$. We have $\sum_{z \in Z'}\elem(M_1|J_z) < \beta_1 q^{r(M_1)} \le \elem(M_1\del X)$; let $y$ be a nonloop of $M_1\del X$ such that $y \notin J_z$ for all $z \in Z$.  
			
		Let $J = J_y$. We now argue that no nonloop of $(M_1 \con J)|(Z'-J)$ is in a short line. Suppose otherwise; let $\cl_{M_1 \con J}(\{e,f\})$ be a short line of $M_1 \con J$ with $e \in Z'$. The line $\cl_{M_1}(\{e,f\})$ is also short in $M_1$. Since $f \notin J$ we know that the line $\cl_{M_1}(\{y,f\})$ is not short, so there is some triangle $\{y,f',f\}$ of $M$. If $\cl_{M_1}(\{e,f'\})$ is short then $\cl_{M_1}(\{e,f\})$ and $\cl_{M_1}(\{e,f'\})$ are two short lines of $M_1$ whose union spans $y$; this contradicts $y \notin J_e$. Therefore $\cl_{M_1}(\{e,f'\})$ is not short in $M_1$. But $y \in J$ so $f'$ and $f$ are parallel in $M_1 \con Y$; it follows that $\cl_{M_1}(\{e,f\})$ is not short in $M_1 \con J$, contradicting its definition. 
		
		Therefore no short line of $M_1 \con J$ contains a nonloop of $(M_1 \con J)|Z'$. However, using $r(M_1|J) \le q^{{4c^2}}$ and Lemma~\ref{controlclass}, we have
		\begin{align*}
		\elem((M_1 \con J)|(Z'-J)) &\ge q^{-q^{4c^2}-c}\elem(M_1|Z') \\
							     &\ge q^{-q^{5c^2}} \lfloor \beta_1 q^{-q^{5c^2}} q^{r(M_1)}\rfloor\\
							     &\ge \tfrac{1}{2}\beta_1 q^{-2q^{5c^2}}q^{r(M_1 \con J)},
\end{align*}
	 where we use $r(M_1) \ge r(M_1 \con J)$ and the fact that $\beta_1q^{-q^{5c^2}}q^{r(M_1)} \ge \beta_1q^{n_0-c^2s-q^{5c^2}} \ge 2$. Thus, $J$ and $Z = Z' - J$ satisfy the claim.
		\end{proof}

			
		Let $M_2 = M_1 \con (J \cup \{f\})$ for some $f \in E(R)$. Since every rank-$s$ restriction of $M_1$ is $\GF(q)$-representable and $r_{M_1}(J) \le s-4$, every rank-$3$ restriction of $M_2$ is $\GF(q)$-representable; in particular, $M_2$ has no long lines. We now build a nearly spanning projective geometry restriction in $M_2$ using $Z$.
		
		\begin{claim}
			$M_2$ has a $\PG(r(M_2)-t-1,q)$-restriction. 
		\end{claim}
		\begin{proof}
			Note that $(M \con f)|E(R)$ has a projective geometry restriction of rank at least $20(q+2) + c^2s + q^{4c^2}+2$. Since $M_2$ is obtained from $M \con f$ by contracting a set of rank at most $c^2s + q^{4c^2}$, we see that $M_2$ has a $\PG(20(q+2)+2,q)$-restriction. Let $m$ be maximal so that $M_2$ has a $\PG(m-1,q)$-restriction $G$, so $m \ge 20(q+2)+2$; assume that $m < r(M_2) - t$. We have $\elem(M_2 | \cl_{M_2}(G)) \le q^{r(G)+c} < q^{c-t} q^{r(M_2)} \le \beta_2 q^{r(M_2)}$, so there is some  $e \in Z - \cl_{M_1}(E(G))$. Let $G' = M|(\cup_{x \in E(G)} \cl_{M_2}(\{e,x\}))$.
			
			We will now apply Lemma~\ref{recognition} to $G'$. We have $\si(G' \con e) \cong G \in \cP_q(0) = \cP_q(0,10(q+2)+1)$, and $r(G') = m+1 \ge 20(q+2)+3$. By choice of $e$, every line of $G'$ through $e$ contains exactly $q+1$ elements, and if there is a line of $G'$ through $e$ that is not modular to some other line of $G'$, then we can contract a point of the latter line to find a $U_{2,q+2}$-restriction; this contradicts the fact that every rank-$3$ restriction of $M_2$ is $\GF(q)$-representable. Therefore $(L,L')$ is a modular pair for every pair of lines of $G'$ with $e \in L$. Lemma~\ref{recognition} applied with $\ell = q+2$ and $t = 1$ now gives $G'\del D \in \cP_{q}(k)$ for some $k \ge 0$ and $D \subseteq \cl_{G'}(\{e\})$. Since matroids in $\cP_q(k)$ have no coloop, it follows that $G'$ has a spanning restriction in $\cP_q(k)$, and so has a $\PG(r(G')-1,q)$-restriction. This contradicts the maximality of $m$. 
		\end{proof}
		Let $G$ be such a restriction and $B$ be a basis of $M_2$ containing a basis $B_G$ for $G$. Let $C = E(S) \cup J \cup (B-B_G)$. We have $M \con C = M_2 \con (B-B_G)$, so $G$ is a $\PG(r(M \con C)-1,q)$-restriction of $M \con C$. Moreover 
		\[r_M(C) \le r_M(S) + r_{M_1}(J) + |B_G| \le c^2s + (s-3) + t \le f_{\ref{getnearspanningpg}}(q,c),\] as required. 
	\end{proof}
	
\section{Finding a Projection}

	In this section, we show that a very high-rank matroid $M_0$ that is a few contractions away from being a very highly locally representable matroid in $\cP_q(j)$ for some $j$ is either in $\cP_q(k)$ for some $k$, or contains a high-rank minor in $\cP_q(k)$ that is, in a certain sense, denser than $M_0$. 
	
	\begin{lemma}\label{finddenseprojection}
		There is a computable function $f_{\ref{finddenseprojection}}\colon \bZ^5 \to \bZ$ so that, for every prime power $q$ and all integers $c,t,j,m \ge 0$, if $\cM \in \fE_q$ satisfies $c_{\cM} \le c$, and $M_0 \in \cM$ satisfies $r(M_0) \ge f_{\ref{finddenseprojection}}(q,c,t,j,m)$ and $M_0 \con K \in \cP_q(j,f_{\ref{finddenseprojection}}(q,c,t,j,m))$ for some rank-$t$ set $K$ of $M_0$, then there is an integer $k \ge 0$ and a minor $M$ of $M_0$ of rank at least $m$, so that $M \in \cP_q(k)$ and either $M = M_0$, or \[\tfrac{q^{r(M)+k}-1}{q-1} - \elem(M)  < \tfrac{q^{r(M_0)+k}-1}{q-1} - \elem(M_0). \]
	\end{lemma}
	\begin{proof}
		
		Let $c,t,j,m \ge 0$ be integers. Let $\delta = q^{-c-t-4}$. Let $b =\delta^{-1} q^{t+1+c} = q^{2t+2c+5}$, and let $\bI$ be the interval $[-b,b] \subseteq \bZ$.
		
		Let $u = 10(c+2)+t$. Let $h\colon \bI^4 \to \bZ$ be a function so that $h(\mathbf{i}) \ge u$ for all $\mathbf{i} \in \bI^4$ and $h$ is strictly decreasing with respect to the lexicographic order on $\bI^4$. Let $n_0\ge \max(m+1,2c+3t-j+5,2u+t)$ be an integer so that $q^{n_0+k} > q^{2c} + q^{c+t}$,
		\[\delta \tfrac{q^{r+t-j}-1}{q-1} > q^{t+c+1} \text{\ \ and \ \ }
		|2(q^{k+t-j}-1-(q-1)s)| < q^{r-t+j}-1\] for all $r \ge n_0$, all $0 \le k < c$ and all $s$ with $|s| \le q^{t+c+1}+ q^{2c}$.
		
		Let $n\colon \bI^4 \to \bZ$ be a function that is strictly decreasing with respect to the lexicographic order on $\bI^4$, and additionally satisfies $n(\mathbf{i}) \ge n_0$ for all $\mathbf{i} \in \bI^4$, and 
		\[n(i_1,i_2,i_3,i_4) \ge t + f_{\ref{getgoodpg}}(q,j,h(i_1,i_2,i_3,i_4),n(i_1,i_2,i_3,i_4+1))\]
		whenever $(i_1,i_2,i_3,i_4), (i_1,i_2,i_3,i_4+1) \in \bI^4$. Set $f_{\ref{finddenseprojection}}(q,c,t,j,m) = \max_{\mathbf{i} \in \bI^4}(\max(h(\mathbf{i}),n(\mathbf{i})))$.
		
		 Let $\cM \in \fM_q$ be such that $c_{\cM} \le c$. Let $M_0 \in \cM$ be a matroid with $r_0 = r(M_0) \ge f_{\ref{finddenseprojection}}(q,c,t,j,m)$ and $M_0 \con K \in \cP_q(k,f_{\ref{finddenseprojection}}(q,c,t,j,m))$. For each minor $N$ of $M_0$ for which $r_N(K) = t$ and $N \con K \in \cP_q(k,h')$ for some $h' \ge 2$, we define a $4$-tuple $\sigma(N)$ measuring the `symmetry' and `modularity' of $N$ relative to the set $K$. For such a minor $N$, let $G_N$ be a simplification of $N \con K$, let $A_{N,K}(x) = \cl_N(\{x\} \cup K) - \cl_N(K)$, and let $a_{N,K}(x) = \elem(N|A_{N,K}(x))$. The sets $A_{N,K}(x)\colon x \notin \cl_N(K)$ correspond to elements of $G_N$; let  $\mu(N),\nu(N)$ and $\rho(N)$ respectively denote the mean, minimum, and maximum of $a_{N,K}(x)$ over all $x \in E(G_N)$. Let
		\begin{align*}
			\sigma_1(N) &= \left\lfloor \delta^{-1}(\mu(N) - \mu(M_0))\right\rfloor, \\
			\sigma_2(N) &= 
			 \elem(N|\cl_N(K)) - \elem(M_0|\cl_{M_0}(K)),\\
			\sigma_3(N) &= \rho(N) - \rho(M_0),\\
			\sigma_4(N) &= \nu(N) - \nu(M_0), \text{\ and}\\
			\sigma(N) &= (\sigma_1,\sigma_2,\sigma_3,\sigma_4)
		\end{align*}
		 Note that $\nu(N) \le \mu(N) \le \rho(N) \le q^{t+c+1}$ and $\elem(N|\cl_N(K)) \le q^{t+c}$, so $\sigma(N) \in \bI^4$ for any such minor $N$. Let $(i_1,i_2,i_3,i_4)$ be  maximal in the lexicographic order on $\bI^4$ such that there is a minor $M$ of $M_0$ of rank at least $n(i_1,i_2,i_3,i_4)$ for which $r_M(K) = t$, $M \con K \in \cP_q(j,h(i_1,i_2,i_3,i_4))$,  and $\sigma(M,K) \ge_{\lex} (i_1,i_2,i_3,i_4)$. Subject to the choice of $(i_1,i_2,i_3,i_4)$, choose such an $M$ for which $|M|$ is as large as possible.  It is clear from the monotonicity of $h(\cdot)$ and $n(\cdot)$ that $\sigma(M) = (i_1,i_2,i_3,i_4)$. Let $r = r(M)$; we have $r \ge n(i_1,i_2,i_3,i_4)\ge n_0$.
								
		 We argue that $M$ (or some single-element contraction of $M$) is the required matroid by proving a succession of claims that use the maximality of $(i_1,i_2,i_3,i_4)$ to show that $M$ is very highly symmetric and modular with respect to $K$. Let $G$ be a simplification of $M \con K$, so $r(G) = r-t+j$ and $|G| = \frac{q^{r-t+j}-1}{q-1} \ge q^{r-t+j-1}$. Let $Y = \{y \in E(G)\colon a_{M,K}(y) < \rho(M)\}$.
		
		\begin{claim}
			$|Y| < \tfrac{1}{2}|G|$.
		\end{claim}
		\begin{proof}
			Suppose otherwise, so $|Y| \ge \tfrac{1}{2}|G| \ge q^{-1}|G| \ge q^{r-t+j-2}$. 
			
		Let $x_0 \in E(G)$ satisfy $a_{M,K}(x_0) = \rho(M)$. For each $y \in Y$, let $\{y,y',x_0\}$ be a triangle of $G$. The line $\cl_M(\{x,y'\})$ is skew in $M$ to $K$ but not to $K \cup y$, and contains at most one element of $A_{M,K}(y)$. Since $a_{M,K}(x_0) > a_{M;K}(y)$, there is thus some $x_y \in A_{M,K}(x_0)$ such that $\cl_M(\{x_y,y'\})$ does not intersect $A_{M,K}(y)$, from which it follows that $y' \in A_{M \con x_y,K}(y)$ and so $a_{M \con x_y,K}(y) > a_{M,K}(y)$. Since $a_{M,K}(x_0) \le q^{t+c+1}$, there is some $v \in A_{M,K}(x_0)$ and a set $Y' \subseteq Y$ such that $|Y'| \ge q^{-t-c-1}|Y|$ and $x_y = v$ for all $y \in Y'$. Note that $v \notin Y'$. Let $G'$ be a simplification of $G \con v$; note that $G' \in \cP_q(k,h(i_1,i_2,i_3,i_4)-1)$. For each $x \in E(G')$, let $L_x = \cl_G(\{x,v\})-\{v\}$. We have $a_{M \con v,K}(x) = a_{M,K}(f)$ for all $f \in L_x$, so
		\begin{align*}
			a_{M \con v,K}(x) &= q^{-1}\sum_{f \in L_x}a_{M \con f,K}(f) \\
			& \ge q^{-1} \left(\sum_{f \in L_x} a_{M,K}(f) + |Y' \cap L_x|\right).
		\end{align*}
		 Using the fact that the sets $\{L_x\colon x \in E(G')\}$ partition $E(G \del v)$, we can sum over all $x \in E(G')$ to obtain
		\begin{align*}
			\sum_{x \in E(G')} a_{M \con v,K}(x) &\ge q^{-1}\left(\sum_{f \in E(G \del v)} a_{M,K}(f) + |Y'|\right) \\ 
			&= q^{-1} (|G| \mu(M) - a_{M,K}(v)) + q^{-1}|Y'| \\ 
			&\ge |G'| \mu(M) - q^{c+t} + q^{-c-t-2}|Y|\\
			&\ge |G'| \mu(M) - q^{r-2t-c+j-5} + q^{-c-t-2}q^{r-t+j-2}\\
			&=  |G'| \mu(M) + (q-1)\delta q^{r-t-1+j}\\
			&> (\mu(M,K) + \delta)|G'|,		
		\end{align*}
		where we use $|G'| = \frac{q^{r-t+j-1}-1}{q-1} < q^{-1}|G|$ and $\delta = q^{-c-t-4}$, as well as our lower bound for $|Y|$, and the fact that $r \ge n_0$ so $c+t \le r-2t-c+j-5$. 
		
		This gives $\sigma_1(M \con v) > i_1$, and so $\sigma(M \con v) \ge_{\lex}(i_1+1,0,0,0) >_{\lex}(i_1,i_2,i_3,i_4)$. Now we have $r_{M \con v}(K) = t$ and, since $h(i_1,i_2,i_3,i_4)-1 \ge h(i_1+1,0,0,0)$, we also have $(M \con v) \con K \in \cP_q(j,h(i_1+1,0,0,0))$.  Moreover $r(M \con v) = r(M) - 1 \ge n(i_1,i_2,i_3,i_4)-1 \ge n(i_1+1,0,0,0)$ so $M \con v$ gives a contradiction to the maximality of $(i_1,i_2,i_3,i_4)$. 		
	\end{proof}
		
	\begin{claim} $\nu(M) = \mu(M) = \rho(M)$. 
	\end{claim}
	\begin{proof}[Proof of claim:]
		Suppose otherwise, so $\nu(M) \le \rho(M)-1$. Let $X = E(G)-Y$. We have $|X| \ge \tfrac{1}{2}|G|$; since \[r(G) \ge n(i_1,i_2,i_3,i_4)-t \ge f_{\ref{getgoodpg}}(q,j,h(i_1,i_2,i_3,i_4),n(i_1,i_2,i_3,i_4+1)),\]
		Lemma~\ref{getgoodpg} implies that there is an independent set $C$ of $G$ so that $G \con C \in \cP_q(j,h(i_1,i_2,i_3,i_4))$, $r(G \con C) \ge n(i_1,i_2,i_3,i_4+1)$, and each parallel class of $G \con C$ contains an element of $X$. It follows that $r_{M \con C}(K) = t$ and $M \con (C \cup K) \in \cP_q(j,h(i_1,i_2,i_3,i_4+1))$ (using monotonicity of $h(\cdot)$), and every parallel class of $G \con C$ contains some $f \in X$ for which $a_{M \con C,K}(f) \ge a_{M,K}(f) = \rho(M)$. So $\nu(M \con C) \ge \rho(M) > \nu(M)$. This implies that $\sigma_1(M \con C) \ge \sigma_1(M)$, $\sigma_3(M \con C) \ge \sigma_3(M)$ and $\sigma_4(M \con C) > \sigma_4(M)$. Clearly $\sigma_2(M \con C) \ge \sigma_2(M)$; therefore $\sigma(M \con C) \ge_{\lex}(i_1,i_2,i_3,i_4+1)$, so $M \con C$ gives a contradiction to the maximality of $(i_1,i_2,i_3,i_4)$.
	\end{proof}
	\begin{claim}\label{kmodularity}
		For every line $L$ of $M$ and every $f \in E(M)- \cl_M(K)$, the pairs $(\cl_M(K),L)$ and $(\cl_M(K \cup \{f\}),L)$ are modular in $M$. 
	\end{claim}
	\begin{proof}
		If one of these pairs is not modular, then there is a line $L$ of $M$ such that either
		\begin{enumerate}[(a)]
			\item\label{topnonmodular} $L \cap \cl_M(K) = \varnothing$ and $\sqcap_M(L,K) = 1$, or
			\item\label{sidenonmodular} $\sqcap_M(L,K) = 0$ and $L \cap \cl_M(K \cup \{f\}) = \varnothing$ and $\sqcap_M(L,K \cup \{v\}) = 1$ for some $v \in E(M) - \cl_M(K \cup \{v\})$. 
		\end{enumerate}
		In either case, there is some $u \in L - \cl_M(K)$. Now $r_{M \con u}(K) = t$ and $(M \con u) \con K \in \cP_q(j,h(i_1,i_2,i_3,i_4)-1)$. Since $a_{M \con u,K}(f) \ge a_{M,K}(f) = \mu(M)$ for each $f \in E(G) - \{u\}$, we have 
		\[\rho(M \con u) \ge \mu(M \con u) \ge \mu(M) = \rho(M),\]
		 which implies that $\sigma_1(M \con u) \ge \sigma_1(M)$ and $\sigma_3(M \con u) \ge \sigma_3(M)$. It is clear that $\sigma_2(M \con u) \ge \sigma_2(M)$. Moreover, we have $\elem(M \con u|\cl_{M \con u}(K)) > \elem(M|\cl_M(K))$ in case (\ref{topnonmodular}) and $\rho(M \con u) \ge a_{M \con u,K}(v) > a_{M,K}(v) = \rho(M)$ in case (\ref{sidenonmodular}), so either $\sigma_2(M) > i_2$ or $\sigma_3(M) > i_3$. In either case, $\sigma(N) \ge_{\lex} (i_1,i_2,i_3+1,0)$.  As in the proof of the first claim, we use $h(i_1,i_2,i_3+1,0) > h(i_1,i_2,i_3,i_4)$ and $n(i_1,i_2,i_3,i_4) > n(i_1,i_2,i_3+1,0)$ to obtain a contradiction to the maximality of $(i_1,i_2,i_3,i_4)$. 
	\end{proof}
	
	Recall that $u = 10(c+2)+t$. By the definitions of $h(\cdot)$ and $n(\cdot)$, we have $r(M) \ge 2u+t$ and $M \con K \in \cP_q(j,u)$. It now follows from~\ref{kmodularity} and Lemma~\ref{recognition} that $M \del D \in \cP_q(k)$ for some $k$ and some $D \subseteq \cl_M(K)$. Note that $k < c$.  Choose $D$ so that $|D|$ is as small as possible; it follows from this choice that $D$ is a union of parallel classes of $M|\cl_M(K)$. Let $d$ be an integer so that $\elem(M \del D) = \frac{q^{r+k}-1}{q-1}-qd$, where $0 \le d \le \tfrac{q^{2k}-1}{q^2-1} < q^{2c-1}$.
	
	\begin{claim}
		$\mu(M) = q^{t+k-j}$ and $\elem(M|(\cl_M(K)-D)) = \frac{q^{t+k+j}-1}{q-1}-qd$. 
	\end{claim}
	\begin{proof}
	We have $\elem(M) = \frac{q^{r+k}-1}{q-1} - qd + \elem(M|D)$. Let $d' = \elem(M|D)$ and $d'' = \elem(M|(\cl_M(K)-D))$. Now
	\begin{align*}
		\tfrac{q^{r+k}-1}{q-1}-qd &= \elem(M \del D) \\ 
		&= \elem(M \con K)\mu(M) + \elem(M|(\cl_M(K)-D))\\
		&= \tfrac{q^{r-t+j}-1}{q-1}\mu(M) + d''.
	\end{align*}
	Rearranging gives 
	\begin{align*}
		\mu(M) &= \tfrac{q^{r-k}-1}{q^{r-t+j}-1} - \tfrac{(q-1)(d''+d)}{q^{r-t+j}-1}\\
		&= q^{k+t-j} - \tfrac{q^{k+t-j}-1-(q-1)(d''+d)}{q^{r-t+j}-1}.
	\end{align*}
	Since $|d''+d| \le q^{c+t+1} + q^{2c}$ and $r \ge n_0$, the second `error' term has absolute value less than $\tfrac{1}{2}$; since $\mu(M) = \rho(M) \in \bZ$, it follows that $\mu(M) = q^{k+t-j}$. Therefore the error term is exactly zero, giving $(q-1)(d''+d) ={q^{k+t+j}-1}$. The claim now follows.
	\end{proof}
	
	 Since $i_1 = \sigma_1(M)$, we have $\mu(M_0) \le \mu(M) - \delta i_1$, so 
	\begin{align*}
		\elem(M_0) &= \mu(M_0)\elem(M_0 \con K) + \elem(M_0 | \cl_{M_0}(K)) \\
		&\le (q^{k+t-j} - \delta i_1)\tfrac{q^{r_0-t+j}-1}{q-1} + d' +d'' -i_2 \\
		&= \tfrac{q^{r_0+k}-1}{q-1} - \tfrac{q^{k+t-j}-1}{q-1} - \delta i_1\tfrac{q^{r_0+t-j}-1}{q-1} + d'+ d''-i_2\\
		&= \tfrac{q^{r_0+k}-1}{q-1} -qd + d'' -\delta i_1\tfrac{q^{r_0+t-j}-1}{q-1} - i_2.\\
		&= \tfrac{q^{r_0+k}-1}{q-1}-qd + d'' - i_1 q^{t+c+1} - i_2,
		\end{align*}
		where we use $r_0 \ge n_0$. Since $(i_1,i_2,i_3,i_4) \ge_{\lex} (0,0,0,0)$, we have $i_1 \ge 0$, and either $i_2 \ge 0$ or $i_1 > 0$. Now $i_2 \le q^{t+c+1}$, so in either case we have $q^{c+t+1} i_1 + i_2 \ge 0$, with equality only if $i_1 = i_2 = 0$. Thus 
		\[\elem(M_0) \le \tfrac{q^{r_0+k}-1}{q-1} - qd + d'',\]
		with equality only if $i_1 = i_2 = 0$ and $\mu(M_0) = \mu(M)$. We now claim that either $D$ is empty, or another matroid $M'$ satisfies the lemma's conclusion:
	
	\begin{claim}
	Either $D = \varnothing$, or $M_0$ has a minor $M'$ of rank at least $m$ such that $M' \in \cP_q(k+1)$ and $\frac{q^{r(M')+k+1}-1}{q-1} - \elem(M') < \frac{q^{r(M_0)+k+1}-1}{q-1} - \elem(M_0)$.
	\end{claim}
	\begin{proof}[Proof of claim:]
	
		Suppose that $D \ne \varnothing$. Note that $D$ contains no loops by its minimality; let $x \in D$ and $M' = M \con x \del (D-x)$. Since $M \del D \in \cP_q(k)$ and $M \del (D-x) \notin \cP_q(k)$, we have $M' \in \cP_q(k+1)$, and $k+1 < c$, so $\elem(M') \ge \frac{q^{(r-1)+k+1}-1}{q-1} - q^{2c}$. Now
		\begin{align*}
			\elem(M_0) - \elem(M') &\le \tfrac{q^{r_0+k}-1}{q-1} + q^{t+c} - \left(\tfrac{q^{r+k}-1}{q-1} - q^{2c}\right) \\
			&= \tfrac{q^{k+1}}{q-1}(q^{r_0}-q^{r-1}) + q^{2c} + q^{c+t} - q^{r_0 + k}\\
			&< \tfrac{q^{k+1}}{q-1}(q^{r_0}-q^{r(M')}),
		\end{align*}
		since $q^{r_0+k} \ge q^{n_0+k} > q^{2c} + q^{c+t}$. Now, since $M' \in \cP_q(k+1)$ and $r(M') = r-1 \ge n(i_1,i_2,i_3,i_4)-1 \ge m$, by rearranging the calculation above we see that $M'$ satisfies the claim. 	\end{proof}
	
	We may thus assume that $D = \varnothing$ and so $d'' = 0$ and $M \in \cP_q(k)$. This gives $\elem(M) = \frac{q^{r+k}-1}{q-1} - qd$ and $\elem(M_0) \le \frac{q^{r_0+k}-1}{q-1}-qd$ with equality only if $i_1 = i_2 = 0$ and $\mu(M) = \mu(M_0)$. If equality does not hold, then $\frac{q^{r+k}-1}{q-1} - \elem(M) < \frac{q^{r_0+k}-1}{q-1} - \elem(M_0)$, so $M$ satisfies the lemma. Assume that equality holds; since $(0,0,0,0) \le_{\lex} (i_1,i_2,i_3,i_4) = (0,0,i_3,i_4)$ and $\mu(M) = \rho(M)$, we have
	\[0 \le i_3 = \rho(M)-\rho(M_0) \le \mu(M) - \mu(M_0) = 0, \]
	so $\rho(M_0) = \mu(M_0) = \mu(M)$. Therefore $\rho(M_0) = \nu(M_0) = \mu(M)$, so $\sigma(M) = (i_1,i_2,i_3,i_4) = (0,0,0,0)$, and now $M = M_0$ by the maximality of $|M|$.
	\end{proof}

\section{The Main Result}

For each prime power $q$, let \[D_q = \{(k,d) \in \bZ_{\ge 0}^2\colon 0 \le d \le q\tfrac{q^{2k}-1}{q^2-1}\}.\] For each $(k,d) \in D_q$, let $\cP_q^d(k)$ be the set of matroids $M$ in $\cP_q(k)$ satisfying $\elem(M) = \frac{q^{r(M)+k}-1}{q-1}-d$. (Note that $\cP_q^d(k) = \varnothing$ unless $d$ is a multiple of $q$.) Define an ordering $\prec$ on $D_q$ by $(k,d) \prec (k',d')$ if and only if $(k,-d) <_{\lex} (k',-d')$. If $(k,d) \prec (k',d')$, then the matroids in $\cP_q^{d'}(k')$ are `denser' than those in $\cP_q^d(k)$. 

 We now prove a technical lemma that combines Lemmas~\ref{getnearspanningpg} and~\ref{finddenseprojection}, and will easily imply our main theorem. 

\begin{lemma}\label{maintechcompute}
	There is a computable function $f_{\ref{maintechcompute}} \colon \bZ^3 \to \bZ$ so that, for each prime power $q$ and all integers $c,m \ge 0$, if $\cM \in \fE_q$ satisfies $c_{\cM} \le c$, and $M \in \cM$ is such that $r(M) \ge f_{\ref{maintechcompute}}(q,c,m)$ and $\elem(M) \ge \tfrac{q^{r(M)+k}-1}{q-1}-d$ for some $(k,d) \in D_q$, then either 
	\begin{itemize}
	\item $M \in \cP_q(k)$, or
	\item there exists $(k',d') \in D_q$ with $(k',d') \succ (k,d)$ and a minor $M'$ of $M$ with $r(M') \ge m$ and $M' \in \cP_q^{d'}(k')$. 
	\end{itemize}
\end{lemma}
\begin{proof}
	Let $q$ be a prime power, and $c,m \ge 0$. Define a sequence $h_0, \dotsc, h_c$ recursively by 
	\[h_\ell = \max \left\{f_{\ref{finddenseprojection}}(q,c,t,j,m)\colon 0 \le j \le c,\ \  0 \le t \le m+c + \sum_{i=0}^{\ell-1} h_i \right\}\]
	for each $\ell \in \{0, \dotsc, c\}$. (The summation is zero for $\ell = 0$.) Let $n_2 = \max_{0 \le i \le c}h_i$. Let $n_1 = \max(c,f_{\ref{mccdhj}}(q,c,q^{-1},n_2))$ and let $n_0 = \max(m,f_{\ref{roundness}}(q,c,c+1,m))$. Set $f_{\ref{maintechcompute}}(q,c,m) = n_0$.

	Let $\cM \in \fE_q$ be such that $c_{\cM} \le c$. Write $g(n)$ for $\frac{q^{n+k}-1}{q-1} - d$; note that, since $k < c$ and $(k,d) \in D_q$, we have $g(n) > q g(n-1) \ge q^{n-1}$ for all $n \ge c$.  Let $M_0 \in \cM$ satisfy $r(M_0) \ge n_0$ and $\elem(M_0) \ge g(r(M_0))$.  By Lemma~\ref{roundness}, we see that $M_0$ has a weakly round restriction $M_1$ such that $r(M_1) \ge n_1$ and either $M_1 = M_0$ or $\elem(M_1) \ge g(r(M_1))+1$. 
	
	 By Theorem~\ref{mccdhj}, $M_1$ has an $\AG(n_2-1,q)$-restriction $R$. Let $M_2$ be a minimal contraction-minor of $M_1$ such that 
	\begin{itemize}
		\item $R$ is a restriction of $M_2$, and 
		\item either $M_2 = M_0$ or $\elem(M_2) \ge g(r(M_2))+1$. 	
	\end{itemize}
	Note that $r(M_1) \ge r(R) > c$, so $g(r(M_1)) > q g(r(M_1)-1)$. By minimality of $M_2$, each $e \in E(M_1) - \cl_M(E(R))$ thus satisfies 
	\[\elem(M_2 \con e) \le g(r(M_2 \con e)) < q^{-1}g(r(M_2)) \le q^{-1} \elem(M_2).\]
	
	 Since $M_1$ is weakly round, so is $M_2$; Lemma~\ref{getnearspanningpg} implies that there is a set $C_1 \subseteq E(M_2)$ so that $r_{M_2}(C_1) \le t_1$ and $M_2 \con C_1$ has a spanning projective geometry restriction $G$. Note that $r(G) \ge n_2 - t_1 \ge m+2c$; let $M_2' = M_2 \con C_1$. Let $C_2$ be a maximal flat of $M_2'$ disjoint from $E(G)$ and let $s = r_{M_2'}(C_2)$. The maximality of $C_2$ implies that $M_2 ' \con C_2 \in \cP_q(s)$, so we have $s < c$. 
	 
	 By maximality of $C_2$, every element of $M_2'$ is spanned by $C_2 \cup \{e\}$ for some $e \in E(G)$, and so $(M_2' \con C_2)|E(G) \in \cP_q(j)$, as $C_2$ spans no element of $G$. Let $j$ be a minimal nonnegative integer so that there exists $C_3 \subseteq E(G)$ for which $r_{M_2'}(C_3) \le \sum_{i=0}^{s-j-1} h_i$ and $r_{M_3 \con C_3}(C_2) = j$. Clearly $0 \le j \le s$. If $j > 0$, then the minimality of $j$ implies that every flat of $G \con C_3$ of rank at most $h_{s-j}$ is skew to $C_2$ in $M_2' \con C_3$; since $h_j \ge 1$, this implies that $M_2' \con (C_2 \cup C_3) \in \cP_q(s,h_{s-j})$. If $j = 0$ then $M_2' \con (C_2 \cup C_3) \in \cP_q(0) = \cP_q(0,h_s)$, so the same conclusion holds. Let $K = C_1 \cup C_2 \cup C_3$ and $t = r_{M_2}(K).$ We have 
	 \[t \le r_{M_2}(C_1) + r_{M_2'}(C_2) + r_{M_2'}(C_3) \le m + c +  \sum_{i = 0}^{s-j-1}h_i.\]
	 Therefore $K$ is a rank-$t$ set in $M_2$ such that $M_2 \con K \in \cP_q(s,h_{s-j})$, where $ f_{\ref{finddenseprojection}}(q,c,t,j,m) \le h_{s-j} $ by the definition of $h_{s-j}$. Now $r(M_2) \ge r(G) = m_1 \ge h_{s-j} \ge f_{\ref{finddenseprojection}}(q,c,t,j,m)$; it follows from Lemma~\ref{finddenseprojection} that $M_2$ has a minor $M_3 \in \cP_q(k')$ for some $k' \in \bZ$, such that $r(M_3) \ge m$ and either $M_3 = M_2$ or 
	 \[\tfrac{q^{r(M_3)+k'}-1}{q-1} - \elem(M_3) < \tfrac{q^{r(M_2)+ k'}}{q-1} - \elem(M_2).\] 
	 Let $\elem(M_3) = \frac{q^{r(M_3)+k'}-1}{q-1}-d'$. If $k' < k$ then the left hand side above is $d' \ge 0$ and the right hand side is at most $\tfrac{q^{r(M_2)+k'}-q^{r(M_2)+k}}{q-1} +d \le d-q^{r(M_2)+k-1} < q^{2k} - q^{c+k} \le 0$, a contradiction. If $k' > k$, then $(k',d') \succ (k,d)$, so $M'= M_3$ satisfies the lemma; assume that $k = k'$. 
	 
	 If $M_3 \ne M_2$, then the above inequality gives $d' < d$, so $(k',-d') \succ (k,-d)$; again, $M' = M_3$ will do. If $M_3 = M_2$, then we either have $M_3 = M_0$ (in which case $M_0 \in \cP_q(k)$ and the first outcome holds) or \[\tfrac{q^{r(M_3)+k}-1}{q-1} - d' = \elem(M_3) \ge g(r(M_3))+1 = \tfrac{q^{r(M_3)+k}-1}{q-1} - (d-1).\]
	 Therefore $0 \le d' < d$, so  $(k,d') \in D_q$ and $(k,d') \succ (k,d)$, so $M' = M_3$ satisfies the lemma. 
	\end{proof}
	
	We now use Lemma~\ref{maintechcompute}  to prove a slightly stronger version of our main result, Theorem~\ref{main}. To see that the statement below implies Theorem~\ref{main}, observe that every $\cM \in \fE_q$ contains the class $\cP_q^0(0)$ of all $\GF(q)$-representable matroids, but is disjoint from $\cP_q(k)$ for all $k \ge c_{\cM}$. It follows easily from maximality that the integers $c,k,d_0,m$ all exist for $\cM$. The advantage of the version stated below is that it gives a computable bound on when the `sufficiently large' condition in Theorem~\ref{main} comes into effect, provided $q,c_{\cM}$ and $m$ are known; this will be useful in the next section.
		
\begin{theorem}\label{maintech}
	There is a computable function $f_{\ref{maintech}}\colon \bZ^3 \to \bZ$ so that, for every prime power $q$ and all integers $c,m \ge 0$, if $\cM \in \fE_q$ satisfies $c_{\cM} \le c$, and $(k,d_0) \in D_q$ is such that $\cM \cap \cP_q^{d_0}(k)$ contains matroids of arbitrarily large rank, but $\cM \cap \cP_q^{d'}(k')$ contains no matroid of rank at least $m$ for any $D_q \ni (k',d') \succ (k,d_0)$, then
	\begin{itemize}
	\item $d_0 = qd$ is a multiple of $q$, 
	\item $h_{\cM}(n) = \frac{q^{n+k}-1}{q-1}-qd$ for all $n \ge f_{\ref{maintechcompute}}(q,c,m)$, and
	\item  for all $M \in \cM$ such that $\elem(M) = \frac{q^{r(M)+k}-1}{q-1}-qd$ and $r(M) \ge f_{\ref{maintech}}(q,c,m)$, we have $M \in \cP_q^{qd}(k)$.
	\end{itemize} 
\end{theorem}	
\begin{proof}
	For each prime power $q$ and all integers $c,m \ge 0$, set $f_{\ref{maintech}}(q,c,m) = n_0 =  \max(q^{58c^4},f_{\ref{maintechcompute}}(q,c,m))$. 
	
	Let $\cM \in \fE_q$ be such that $c_{\cM} \le c$ and let $m \ge 0$ and $(k,d_0) \in D_q$ satisfy the conditions in the hypothesis for $M$. Note that $k \le c$. By Lemma~\ref{genericpoint} and the fact that $\cM \cap \cP_{q}^{d_0}(k)$ contains matroids of arbitrarily large rank, we have $d_0 = qd$ for some $d$, and $h_{\cM}(n) \ge \frac{q^{n+k}-1}{q-1} - qd$ for all $n \ge q^{58k^4}$, and therefore for all $n \ge n_0$. If $r(M) \ge n_0$ and $\elem(M) \ge \frac{q^{r(M)+k}-1}{q-1}-qd$, then Lemma~\ref{maintech} implies that either $M \in \cP_{q}^{qd}(k)$, or there is some $(k',d') \succ (k,qd)$ for which $M$ has a minor $M' \in \cP_{q}^{d'}(k')$ of rank at least $m$. The latter outcome contradicts the choice of $k,d_0,m$, so $M \in \cP_q^{qd}(k)$ for any such $M$. Therefore $h_{\cM}(n) = \frac{q^{n+k}-1}{q-1} - qd$ for all $n \ge n_0$, and any matroid in $\cM$ of rank $n \ge n_0$ whose number of points attains this function is in $\cP_{q}^{qd}(k)$. This gives the theorem.
\end{proof}

\section{Computability}\label{computabilitysection}

	In this section we will prove Theorem~\ref{compute}. Our first step is a technical result that shows that, if a class satisfies three particular conditions then its growth rate function can be determined completely with a finite computation and access to a membership oracle. The third condition is that $\cM$ is closed under a particular type of modular sum; essentially, given a matroid $M \in \cP_{q,k} \cap \cM$ and a spanning projective geometry $G$, we should be able to `extend $G$ into larger rank' and remain in $\cM$. 

\begin{lemma}\label{computability}
	Let $q$ be a prime power, and let $\cM \in \fE_q$ be a class for which there are integers $\ell,b,s \ge 0$ such that \begin{itemize}
	\item $U_{2,\ell+2} \notin \cM$, 
	\item $\PG(s-1,q') \notin \cM$ for all $q' \in \{q+1, \dotsc, \ell\}$, and 
	\item For all $k \ge 0$, if $M \in \cM \cap \cP_q(k)$ and $G \cong \PG(n-1,q)$ are matroids such that $r(M) \ge b$, $E(G) \cap E(M) = F$ and $G|F = M|F \cong \PG(r(M)-1,q)$, then $G \oplus_m M \in \cM$. 
	\end{itemize}
	Then there are integers $k,d,n_0$, all computable given $q,\ell,b$ and $s$ and a membership oracle for $\cM$, such that $h_{\cM}(n) = \frac{q^{n+k}-1}{q-1}-qd$ for all $n \ge n_0$, and so that every matroid $M \in \cM$ with $r(M)  = r \ge n_0$ and $\elem(M) = h_{\cM}(r)$ satisfies $M \in \cP_q^{qd}(k)$.
\end{lemma}
\begin{proof}
	Let $c = c_{\cM}$, noting that $c$ is computable from $q,\ell$ and $s$ by Lemma~\ref{controlclass}. Let $m = 58c^4$ and let $n_0 = f_{\ref{maintech}}(q,c,m)$. 
	
 Let $(k,qd) \in D_q$ be maximal with respect to $\prec$ such that $\cM$ contains a simple rank-$m$ matroid in $\cP_q^{qd}(k)$. Note that $k$ and $d$ can be determined with at most $2^{2^{q^{m+c}}}$ queries to a membership oracle for $\cM$, since every simple rank-$m$ matroid in $\cP_q(k)$ has at most $q^{m+k} \le q^{m+c}$ elements.	\begin{claim}
		$h_{\cM}(n) \ge \frac{q^{n+k}-1}{q-1}-qd$ for all $n \ge m$. 
	\end{claim}
	\begin{proof}[Proof of claim:]
		We may assume that $k > 0$. Let $M$ be a $k$-element projection of $\PG(m+k-1,q)$ so that $\si(M) \in \cM \cap \cP_q^{qd}(k)$. Let $\wh M$ be a rank-$(m+k)$ matroid and $K$ be a $k$-element independent set of $\wh M$ such that $\wh M \del K \cong \PG(m+k-1,q)$ and $\wh M \con K = M$. Let $F$ be a rank-$m$ flat of $\wh M$ that is skew to $K$ (so $M|F = \wh M |F \cong \PG(m-1,q)$), let $G \cong \PG(n+k-1,q)$ be a matroid with $\wh{M}|F$ as a restriction such that $E(G) \cap E(\wh M) = F$, and let $N = \wh M \oplus_{F} \wh G$. Now $\si(N \con K) \cong \si(M \oplus_F G) \in \cM$, since $m \ge b$.
		
		Let $J$ be a maximal independent set of $\wh M \del K$ that is skew to $F$ in $\wh M$, such that $M \con J \in \cP_q^{qd}(k)$. By skewness to $F$ we have $|J| \le k$; if $|J| < k$ then by maximality, every element of $x$ of $\wh M$ is in $\cl_{\wh M}(J \cup F)$, or satisfies $M \con x \notin \cP_q^{qd}(k)$. $\wh M$ has at most $\tfrac{q^{m+k-1}-1}{q-1}$ elements of the first type, and at most $q^{58c^4}$  of the second type by Lemma~\ref{projectiondensity}, but
		\[|\wh M| \ge \tfrac{q^{m+k}-1}{q-1} = \tfrac{q^{m+k-1}-1}{q-1} + q^{m+k-1} \ge \tfrac{q^{m+k-1}-1}{q-1} + q^{58c^4}, \] 
		a contradiction. So $|J| = k$. Now $M \con J \in \cP_q^{qd}(k)$ and $\wh M \con J$ is a matroid in which every element outside $K$ is parallel to an element of $F$; it follows that $(\wh M \con J)|(F \cup K)$ is a simplification of $\wh M \con J$, and $K$ is a rank-$k$ independent set of $\wh M \con J$ spanning no element of $F$. Moreover, $N \con J = G \oplus_m (M \con J)$; by modularity, the set $K$ spans no element of $G$ in $N \con J$, so no two elements of $G \del F$ are parallel in $ N \con (J \cup K)$. Thus
		\begin{align*}
			\elem(\wh N \con (J \cup K)) &= |G \del F| + \elem(\wh M \con (J \cup K))\\
			&= \tfrac{q^{n+k}-q^m}{q-1} + \left(\tfrac{q^{(m-k)+k}-1}{q-1} - qd\right)\\
			&= \tfrac{q^{n+k}-1}{q-1} - qd.
		\end{align*}
		Since $\si(N \con K) \in \cM$ and $r(N \con K) = n+k$, the matroid $\si(N \con (J \cup K))$ is a rank-$n$ matroid in $\cM$, and the claim follows. 
	\end{proof}
	
	Now, if $M \in \cM$ satisfies $r(M) \ge n_0$ and $\elem(M) \ge \frac{q^{r(M)+k}-1}{q-1}-qd$, then either $M \in \cP_q^{qd}(k)$, or $M$ has a minor of rank at least $m$ in $\cP_q^{d'}(k')$ for some $(k',d') \succ (k,d)$. Since $m = q^{58c^4} \ge q^{58(k')^4}$, the former possibility would imply by Lemma~\ref{projectiondensity} that $M$ has a rank-$m$ minor in $\cP_q^{d'}(k')$, contradicting the maximality of $(k,d)$. Therefore $M \in \cP_q^{qd}(k)$ for all such $M$; this implies the lemma. 
\end{proof}

The next lemma shows that, when we exclude some truncation of a projective geometry as a minor, the class of matroids without a given minor is closed under the modular sum operation of Lemma~\ref{computability}.

\begin{lemma}\label{sumexcludeminor}
	Let $q$ be a prime power, let $n,r,t,k \ge 0$ be integers with $n \ge r \ge k(t+1) + r(N)$, and let $N$ be a simple matroid. If $M$ is a rank-$r$ matroid with no $T(\PG(t,q))$-minor and no $N$-minor such that $M \in \cP_q(k)$, and $G \cong \PG(n-1,q)$ satisfies $G|Y = M|Y \cong \PG(r-1,q)$, where $Y = E(G) \cap E(M)$, then $G \oplus_m M$ has no $N$-minor. 
\end{lemma}
\begin{proof}
	Let $K$ be a rank-$k$ independent flat of a matroid $\wh M$ so that $M = \wh M \con K$ and $\wh M \del K \cong \PG(r+k-1,q)$. For each $x \in K$, let $F_x$ be the unique minimal flat of $\wh M \del K$ spanning $x$. We have $(\wh M \con x)|F_x \cong T(\PG(r_{\wh M}(F_x)-1,q))$, and so $M|F_x$ has a $T(\PG(r_M(F_x)-1,q))$-restriction, implying that $r_M(F_x) \le t$ for each $x \in K$; therefore $r_{\wh M}(F_x) \le t+k$. If $F$ is the flat of $\wh M \del K$ spanned by $\cup_{x \in K}F_x$, then we thus have $r_{\wh M}(F) \le k(t+k)$ and $K \subseteq \cl_{\wh M}(F)$; note that $r_{\wh M}(F) \le r(\wh M) - r(N)$ by hypothesis.
	
	It follows from the construction of $F$ that 
	\[\wh M = (\wh M \del K) \oplus_m (\wh M|(K \cup F)) .\]
	Let $\wh G \cong \PG(n+k-1,q)$ be a matroid so that $\wh M \del K$ is a restriction of $\wh G$ to a rank-$(r+k)$ flat, and $G$ is a restriction of $\wh G$ to a rank-$n$ flat that intersects $E(\wh M)$ in a rank-$r$ flat that is skew to $K$ in $\wh M$. By choice of $\wh G$, we have 
	\[G \oplus_m Y \cong ((\wh G \oplus_m (\wh M|(K \cup F)) ) \con K)|(E(M) \cup E(G)). \]
	It therefore suffices to show that, if $M' = \wh G \oplus_m \wh M|(K \cup F)  $, then $M' \con K$ has no $N$-minor. Suppose that $N \cong (M' \con C \del D) \con K$, where $C$ is independent and $D$ is co-independent in $M'$. Now $r(N) = r(M') - |C| - r_{M'}(K)$ and so $|C| \ge n - k - r(N)$. Moreover, $|C \cap F| \le r_{M'}(F) \le kt$, so $|C-F| \ge n-k-r(N)-kt \ge n-r$, so $C$ contains an 
$(n-r)$-element independent set $C'$ of $M'$ that is skew to $F$ in $M'$. However $\si(M' \con C') \cong \si(\wh M)$, and $M = \wh M \con K$ has no $N$-minor, contradicting the fact that $(M' \con C') \con K$ has an $N$-minor.\end{proof}

	The next result, proved in ([\ref{sqf}], Theorem 3.4) will easily imply that matroids representable over given field are also closed under the modular sum operation. 
	
	\begin{theorem}\label{subfieldrepresentation}
		Let $\bF$ be a field with a $\GF(q)$-subfield. If $M$ is an $\bF$-representable matroid with a $\PG(n-1,q)$-restriction for some $n$, then $M$ has an $\bF$-representation so that every column in this restriction has entries only in $\GF(q)$.  
	\end{theorem}

We now restate and prove Theorem~\ref{compute}.

	\begin{theorem}
		Let $\cF$ be a finite set of finite fields and $\cO$ be a finite set of simple matroids. Let $\cM$ be the class of matroids in $\Ex(\cO)$ that are representable over all fields in $\cF$. If $\cM$ is base-$q$ exponentially dense and does not contain all truncations of $\GF(q)$-representable matroids, then there are computable nonnegative integers $k,d$ and $n_0$ such that $h_{\cM}(n) = \frac{q^{n+k}-1}{q-1} - qd$ for all $n \ge n_0$. 
	\end{theorem}
	\begin{proof}
		Since $\cM$ is base-$q$ exponentially dense, there is some matroid $U_{2,\ell+2}$ that is either in $\cO$, or not representable over a field in $\cF$; clearly $\ell$ is computable given $\cF$ and $\cO$. Now $q$ is the largest prime power so that $\GF(q)$ is a subfield of all fields in $\cF$ and  $\cO$ contains no $\GF(q)$-representable matroid. If $\cF \ne \varnothing$, then for each $\bF \in \cF$, the matroid $T(\PG(|\bF|,|\bF|))$ has a $U_{|\bF|,|\bF|+2}$-restriction so is not in $M$. If $\cF = \varnothing$ then $\cO$ must contain some minor of a truncation of a projective geometry over $\GF(q)$ and must therefore contain a spanning restriction of such a truncation. Therefore $T(\PG(t,q)) \notin \cM$, where $t$ is either the minimum size of a field in $\cF$, or the maximum rank of a matroid in $\cO$ if $\cF$ is empty. Finally, since $\cM$ is base-$q$ exponentially dense, it is easy to check that we have $\PG(s-1,q') \notin \cM$ for every prime power $q' \in \{q+1,\dotsc,\ell\}$, where $s$ is the maximum rank of a matroid in $\cO$. 
		
		By Lemma~\ref{controlclass}, we can compute the constant $c_{\cM}$. To show that $n_0$, $k$ and $d$ are computable for $\cM$, it suffices to show that there exists $b \ge 0$ so that $\cM$ is closed under the modular sum operation in Lemma~\ref{computability}. Since each $\bF \in \cF$ has a $\GF(q)$-subfield, it is clear by Lemma~\ref{subfieldrepresentation} that we can adjoin some $\bF$-representation of a matroid $M \in \cM$ to a $\GF(q)$-representation of a projective geometry $G$ to obtain an $\bF$-representation of their modular sum, so the sum operation preserves $\bF$-representability. Since $k \le c_{\cM}$, the fact that $\cM$ is itself closed under the sum operation for $b = \max\{r(N) + c_{\cM}(t+1)\colon N \in \cO\}$,
follows from Lemma~\ref{sumexcludeminor}. Since it is easy to decide membership of a given matroid in $\cM$, Lemma~\ref{computability} now gives the theorem. 
		\end{proof}
		
	Excluding a truncation as a minor is fundamental to the proof of Lemma~\ref{sumexcludeminor} and therefore to Theorem~\ref{compute}. Conjecture~\ref{undecidable} claims that without this exclusion, it is impossible to compute the growth rate function in general for a class defined by excluded minors. Our motivation for this conjecture is the drastic difference between the complexity of describing a $k$-element projection with and without excluding some truncation as a minor, which we now outline. 
	
	The proof of Lemma~\ref{sumexcludeminor} uses the fact that, for $n \ge k(t+k)$, every $k$-element extension of $\PG(n+k-1,q)$ with no $T(\PG(t,q))$-minor is in fact the modular sum of $\PG(n+k-1,q)$ and some $k$-element extension of $\PG(k(t+k)-1,q)$. Each such extension has at most $q^{k(t+k)}$ elements, and thus the number of nonisomorphic $k$-element extensions of $\PG(n+k-1,q)$ with no $T(\PG(t,q))$-minor (and therefore the number of rank-$n$ matroids in $\cP_q(k)$ with no $T(\PG(t,q))$-minor) is at most $2^{2^{q^{k(t+k)}}}$, a bound independent of $n$.
	
	 On the other hand, let $G \cong \PG(n+1,q)$, let $\cF$ be the set of rank-$n$ flats of $G$, and let $\cF' \subseteq \cF$. It is routine to show that there is a unique two-element extension $M = M_{\cF'}$ of $G$ by elements $x_1$ and $x_2$ such that
	 \begin{itemize}
	 \item  $M \del x_i$ is a free extension of $G$ for each $i \in \{1,2\}$,
	 \item  $\{x_1,x_2\}$ is skew to every flat of rank less than $n$ in $G$, and is spanned by no hyperplane of $G$, and
	 \item  the set of rank-$n$ flats of $G$ skew to $\{x_1,x_2\}$ is exactly $\cF'$.
	 \end{itemize}
	 Since different sizes of $\cF'$ correspond to nonisomorphic matroids, this gives at least $|\cF| = \tfrac{(q^{n+2}-1)(q^{n+1}-1)}{(q^2-1)(q-1)} > q^{2n}$ nonisomorphic two-element extensions of $G$. (Actually the number of nonisomorphic $M$ is the number of orbits of the action of $\mathrm{Aut}(G)$ on $2^{\cF}$, which seems likely to be doubly exponential in $n$.) These extensions will correspond to nonisomorphic rank-$n$ matroids in $\cP_q(2)$, and their abundance markedly contrasts the constant number we get when excluding a truncation.
	 
	 The complexity of even these two-element extensions leads us to believe that one can perhaps encode undecidable problems in the form of minor-testing on two-element projections of projective geometries; this motivates Conjecture~\ref{undecidable}.  
\section*{References}
\newcounter{refs}
\begin{list}{[\arabic{refs}]}
{\usecounter{refs}\setlength{\leftmargin}{10mm}\setlength{\itemsep}{0mm}}

\item\label{bry}
T. Brylawski, 
Modular constructions for combinatorial geometries,
Trans. Amer. Math. Soc. 203 (1975), 1--44.

\item\label{gkpaper}
J. Geelen, K. Kabell,
Projective geometries in dense matroids, 
J. Combin. Theory Ser. B 99 (2009), 1--8.

\item\label{gkw}
J. Geelen, J.P.S. Kung, G. Whittle,
Growth rates of minor-closed classes of matroids,
J. Combin. Theory Ser. B 99 (2009), 420--427.

\item\label{oldgrfpaper}
J. Geelen, P. Nelson, 
On minor-closed classes of matroids with exponential growth rate, 
Adv. Appl. Math. 50 (2013), 142--154. 

\item\label{dhjpaper}
J. Geelen, P. Nelson, 
A density Hales-Jewett theorem for matroids, 
J. Combin. Theory Ser. B, to appear. 

\item\label{kungroundness}
J.P.S. Kung, 
Numerically regular hereditary classes of combinatorial geometries,
Geom. Dedicata 21 (1986), no. 1, 85--10.

\item\label{polymath}
D.H.J. Polymath,
A new proof of the density Hales-Jewett theorem,
arXiv:0910.3926v2 [math.CO], (2010) 1-34.

\item\label{sqf}
P. Nelson,
Growth rate functions of dense classes of representable matroids, 
J. Combin. Theory Ser. B 103 (2013), 75--92.

\item \label{oxley}
J. G. Oxley, 
Matroid Theory, Second Edition,
Oxford University Press, New York (2011).
\end{list}		
\end{document}